\documentclass[11pt]{amsart}

\usepackage[arrow,curve,matrix,tips,2cell]{xy}
\usepackage{amssymb,amsthm,amsmath}
\usepackage{hyperref}
\usepackage[alphabetic]{amsrefs}
\usepackage[shortlabels]{enumitem}
\usepackage[margin=3cm]{geometry}
\usepackage{microtype}

\newtheorem{theorem}{Theorem}[section]
\newtheorem{corollary}[theorem]{Corollary}
\newtheorem{lemma}[theorem]{Lemma}
\newtheorem{proposition}[theorem]{Proposition}

\theoremstyle{remark}
\newtheorem{remark}[theorem]{Remark}

\newtheorem{example}[theorem]{Example}

\numberwithin{equation}{section}

\newcommand{\N}{\mathbb{N}}

\newcommand{\Z}{\mathbb{Z}}
\newcommand{\R}{\mathbb{R}}
\newcommand{\C}{\mathbb{C}}
\newcommand{\K}{\mathcal{K}}
\newcommand{\Cu}{\mathrm{Cu}}
\newcommand{\Cusim}{\mathrm{Cu}^\sim}
\newcommand{\cc}{\mathrm{cc}}
\newcommand{\co}{\mathrm{c}}

\newcommand{\ov}[1]{\overline{#1}}
\renewcommand{\epsilon}{\varepsilon}
\renewcommand{\leq}{\leqslant}
\renewcommand{\geq}{\geqslant}
\DeclareMathOperator{\rank}{rank}

\title{A revised augmented Cuntz semigroup}

\author
[Robert]{Leonel Robert}
\address{Leonel Robert
\\Department of Mathematics
\\University of Louisiana at Lafayette
\\Lafayette, 70504-3568
\\USA
} \email{lrobert@louisiana.edu}

\author
[Santiago]{Luis Santiago}
\address{Luis Santiago
\\Department of Mathematical Sciences
\\Lakehead University
\\Thunder Bay
\\Canada
}\email{lsantiag@lakehead.ca}

\begin{document}
\begin{abstract}
We revise the construction of the augmented Cuntz semigroup functor used by the first author to classify inductive limits of 1-dimensional noncommutative CW complexes. The original construction has good functorial properties when restricted to the class of C*-algebras of stable rank one. The construction proposed here has good properties for all C*-algebras: We show that the augmented Cuntz semigroup is  a stable, continuous, split exact functor, from the category of C*-algebras to the category of Cu-semigroups.  
\end{abstract}

\maketitle

\section{Introduction}
The Elliott classification program   purports to distinguish ``well behaved" simple nuclear separable C*-algebras up to isomorphism relying on essentially two kinds of data: K-groups and the cone of traces. This classification program has recently come almost to completion, culminating decades of research. In these recent developments the Cuntz semigroup plays  a key role  as a tool to define and exploit the regularity properties of the classifiable C*-algebras. When looking into classifying non-simple C*-algebras, the Cuntz semigroup  itself becomes a good candidate for a classification invariant. One limitation of the Cuntz semigroup, however, is that it fails to capture the $K_0$-group for non-unital C*-algebras. To remedy this, a variation on the Cuntz semigroup was introduced in \cite{Robert}. This ordered semigroup, which we call here the \emph{augmented Cuntz semigroup}, is built out of the the Cuntz semigroup of the unitization of the C*-algebra,  resembling the way in which the $K_0$-group is built out of the monoid of Murray-von Neumann classes of projections of the unitization of the C*-algebra.  The  augmented Cuntz semigroup was used in \cite{Robert} to classify inductive limits of 1-dimensional noncommutative CW complexes.

A shortcoming of the construction of the augmented Cuntz semigroup given in \cite{Robert} is that it is only a well behaved functor on the class of stable rank one C*-algebras. In the present paper we address this issue by revising the definition of this functor.   The new construction agrees with the old one for C*-algebras of stable rank one (and C*-algebras of finite stable rank). But we are now able to establish the basic properties of the augmented Cuntz semigroup in greater generality.  

Let $\Cusim$ denote the augmented Cuntz semigroup functor. We show that 
\begin{enumerate}
\item	
$\Cusim$ is a functor from the category of C*-algebra to the category of Cu-semigroups (Theorem \ref{CusimCu}),
\item
$\Cusim$ preserves inductive limits (Theorem \ref{Cusimlimits}),
\item
$\Cusim$ is stable and split exact (Theorems \ref{isoembed} and \ref{splitexact}).
\end{enumerate}	
Further, we calculate  $\Cusim(A)$  for $A$ simple, separable, and  pure (Theorem \ref{simpleCusim}). Again, we do this without assuming that $A$ has stable rank one. Under more restrictive hypotheses, although also circumventing the stable rank one hypothesis, this calculation is obtained in \cite[Appendix A]{Linetal}.

This paper is organized as follows: Section \ref{prelims} is devoted to preliminaries on the Cuntz semigroups of C*-algebras and their abstract counterparts, Cu-semigroups. In Section \ref{ccconstruction} we describe a ``formal differences" construction for Cu-semigroups, which we use to define $\Cusim(A)$. In Section \ref{Cusim} we define $\Cusim(A)$ and show, among other results, that it is a Cu-semigroup. In Section \ref{Cusimfunctor} we prove the functorial properties of $\Cusim$ mentioned above. In Section \ref{K0andtraces} we show that if $A$ is stably finite then $K_0(A)$ and the cone of densely finite 2-quasitraces on $A$ can be read off of $\Cusim(A)$. Further, if $A$ is simple and pure, $\Cusim(A)$ can be calculated in terms of $K_0(A)$, the cone of densely finite 2-quasitraces on $A$, and the pairing between these two invariants.

\section{Preliminaries}\label{prelims}

\subsection{The Cuntz semigroup}
Here we briefly recall some facts on the Cuntz semigroup of a C*-algebra and on the category of abstract $\Cu$-semigroups. The reader is referred to \cite{Antoine-Perera-Thiel} and references therein for more on this topic.

Let $A$ be a C*-algebra. Let $A_+$ denote the set of positive elements of $A$. Given $a,b\in A_+$ we say that $a$ is Cuntz smaller than $b$, and denote this by $a\precsim_{\Cu} b$, if there exist
$x_1,x_2,\ldots$ in  $\overline{bA}$  such that $x_n^*x_n\to a$.
We say that $a$ is Cuntz equivalent to $b$, and denote this by $a\sim_{\Cu} b$, if $a\precsim_{\Cu} b$ and $b\precsim_{\Cu} a$. Given $a\in A_+$ we denote by $[a]$ the Cuntz equivalence class of $a$.

Let us denote by $\Cu(A)$ the set $(A\otimes\K)_+/\!\!\!\sim_{\Cu}$ of Cuntz equivalence classes of positive elements of $A\otimes\K$.  This set becomes an ordered semigroup when it is endowed with the order
\[
[a]\le [b]\text{ if } a\precsim_{\Cu} b,
\]
and with the addition operation
\[
[a]+[b]=[a'+b'],
\]
where $a',b'\in (A\otimes\K)_+$ are orthogonal elements that are Cuntz equivalent to $a$ and $b$, respectively (the existence of $a'$ and $b'$ is guaranteed by the stability of $A\otimes\mathcal K$). We call $\Cu(A)$ the Cuntz semigroup of the C*-algebra $A$.

Let $\phi\colon A\to B$ be a C*-algebra homomorphism. Then $\Cu(\phi)\colon \Cu(A)\to \Cu(B)$ is defined as $\Cu(\phi)([a])=[(\phi\otimes \mathrm{id}_{\K})(a)]$ for all $a\in (A\otimes \K)_+$, where  $\mathrm{id}_{\K}\colon \K\to \K$ denotes the identity map. This is a morphism of ordered semigroups.

By \cite[Theorems 1]{Coward-Elliott-Ivanescu}, the assignments $A\mapsto \Cu(A)$ and $\phi\mapsto \Cu(\phi)$ define  a functor from the category of C*-algebras to a certain category  of ordered semigroups which we recall next.

\subsection{$\Cu$-semigroups}\label{Cuprelims}
Let $S$ be an ordered set such that every increasing sequence has a supremum. Given $x,y\in S$ we write $x\ll y$
if whenever $y\leq \sup_n y_n$ for some increasing sequence $(y_n)_n$, then there exists $n_0\in \N$ such that $x\leq y_{n_0}$.  We say then that $x$ is compactly contained in $y$, or way below $y$. An element $e\in S$ is called compact if $e\ll e$. 

Suppose now that  $S$ is an ordered semigroup with neutral element 0. Consider the following properties on $S$:
\begin{description}
\item[O0] 0 is a compact element
\item[O1] Every increasing sequence in $S$ has a supremum.
\item[O2] For each $x\in S$ there exists an $\ll$-increasing sequence 	$(x_n)_n$ such that $x=\sup x_n$.
\item[O3] If $(x_n)_n$ and $(y_n)_n$ are increasing sequences then $\sup x_n+y_n=\sup x_n+\sup y_n$
\item[O4] If $x_1\ll y_1$ and $x_2\ll y_2$ then $x_1+x_2\ll y_1+y_2$.
\end{description}	
In the literature  (\cites{Antoine-Perera-Thiel, Coward-Elliott-Ivanescu}) a Cu-semigroup is typically understood to be a positively ordered semigroup (i.e., one for which $0\leq x$ for all $x\in S$) satisfying  O1-O4. Notice that such groups automatically satisfy O0. We will consider here ordered semigroups that are not necessarily positively ordered.  Instead, by a Cu-semigroup we  shall understand an ordered semigroup that satisfies axioms O0-O4.

By a Cu-morphism we understand a map $\alpha\colon S\to T$ between Cu-semigroups that is an ordered semigroup  morphism (preserves order and addition), preserves  neutral elements, the suprema of increasing sequences,  and the compact containment.  These last two properties mean that
\begin{description}
\item[M1] if $(x_n)_{n=1}^\infty$ is an increasing sequence in $S$  then $\alpha(\sup x_n)=\sup\alpha(x_n)$,
\item[M2] if $x,y\in S$ are such that $x\ll y$ then $\alpha(x)\ll \alpha(y)$.
\end{description}

The category $\Cu$ is defined as the category of Cu-semigroups with Cu-morphisms. If $A$ is a C*-algebra then $\Cu(A)$ is a Cu-semigroup (\cite[Theorem 1]{Coward-Elliott-Ivanescu}).   
By \cite[Theorem 2]{Coward-Elliott-Ivanescu},   $\Cu(\cdot)$ is a functor from the category of C*-algebras to the category $\Cu$.  Moreover, this functor preserves inductive limits.

We will make use below of an additional axiom satisfied by the Cuntz semigroup of a C*-algebra:
\begin{description}
\item[O5] For all $x'\ll x\leq y$ and $w'\ll w$ such that $x+w\leq y$ there exists $z$ such that $x'+z\leq y\leq x+z$
and $w'\ll z$. 		
\end{description}
We some times refer to O5 as the ``almost algebraic order" axiom. That  $\Cu(A)$ satisfies O5 as stated above is proven in \cite{Antoine-Perera-Thiel}. A slightly weaker version, which does not seem to suffice for our purposes, is proven in \cite{Rordam-Winter}.

We call a positive  element $z\in S$ full if $\infty\cdot z:=\sup_n nz$ is the largest element of $S$.

\section{The $S_{\cc}$ construction}\label{ccconstruction}
Throughout this section  $S$ is a positively ordered Cu-semigroup that satisfies the almost algebraic order axiom O5.

Recall that an element $e\in S$ is called compact if $e\ll e$.
We denote the subsemigroup of compact elements of $S$ by  $S_{\co}$, i.e., $S_{\co}=\{e\in S:e\ll e\}$.
Let us  define on $S\times S_c$ a relation as follows: 
$(x,e)\precsim_{\cc} (y,f)$ if for all $x'\ll x$ there exists $g\in S_{\co}$
such that 
\[
x'+f+g\ll y+e+g.
\]
Since  $y$ is the supremum of a $\ll$-increasing sequence,  once this inequality holds we can then choose $y'\ll y$ such that $x'+f+g\leq y'+e+g$. Thus, $(x,e)\precsim_{\cc} (y,f)$ if and only if for all $x'\ll x$ 
there exist $y'\ll y$ and $g\in S_{\co}$ such that $x'+f+g\leq y'+e+g$. Equipped with this observation it is straightforward to check that the relation
$\precsim_{\cc}$ is transitive. Another observation  easily checked, and which we will use frequently,   is that if $x'\precsim_{\cc}y$ for all $x'\ll x$ then 
$x\precsim_{\cc}y$.

By anti-symmetrizing $\precsim_{\cc}$
we obtain an equivalence relation: we write
$(x,e)\sim_{\cc} (y,f)$ if $(x,e)\precsim_{\cc} (y,f)$ and $(y,f)\precsim_{\cc}(x,e)$.
We denote the equivalence class of $(x,e)\in S\times S_{\co}$ by $\overline{(x,e)}$.

Let $S_{\cc}=(S\times S_{\co})/\!\!\sim_{\cc}$. 
Let us endow $S_{\cc}$ with an order and an addition operation. 
The order on $S_{\cc}$ is the one induced by the pre-order $\precsim_{\cc}$:  
\[
\overline{(x,e)}\leq \overline{(y,f)}\hbox{ if $(x,e)\precsim_{\cc} (y,f)$.}
\]
Addition in $S_{\cc}$ is defined in the obvious way:
\[
\overline{(x,e)}+\overline{(y,e)}:=\overline{(x+y,e+f)}.
\]
It is straightforward to check that addition is well defined and compatible with the order. Thus, $S_{\cc}$ is an ordered monoid with neutral element $\overline{(0,0)}$.

Given $x\in S$ let us denote $\ov{(x,0)}$---the equivalence class of $(x,0)$---simply by $\overline{x}$. Let us denote the neutral element $\ov{(0,0)}$ by $0$. 
If $e$ is a compact element then $\overline{e}+\overline{(0,e)}=0$, i.e., $\overline{(0,e)}$ is the additive inverse of $\overline{e}$. We can thus write
$\overline{(x,e)}$ as $\overline x-\overline{e}$. This is the form in which we typically write  the elements of $S_{\cc}$:
\[
S_{\cc}=\{\overline x-\overline e:x\in S,\, e\in S_{\co}\}.
\]

Suppose that $e\in S_{\co}$ is an order unit of the semigroup $S_{\co}$. That is,  for any $f\in S_{\co}$ we have $f\leq ne$ for some $n\in \N$.
If $f\leq ne$, with $f,e$ compact, then  $f+f'=ne$ for some $f'\in S_{\co}$ (a consequence of the almost algebraic order axiom O5). 
It follows that  every element of $S_{\cc}$ can be expressed in the form $\ov{x}-n\ov{e}$ for some $x\in S$ and $n\in \N$.
That is,
\[
S_{\cc}=\{\ov{x}-n\ov{e}: x\in S,\, n\in \N\}.
\]

\begin{example} \label{Ncc}
If $S=\overline{\N}:=\N\cup\{\infty\}$ then $S_{\cc}\cong \ov{\Z}:= \Z\cup \{\infty\}$.
\end{example}

\begin{lemma}\label{ontopositives}
The map $x\mapsto \overline{x}$ from $S$ to $S_{\cc}$  is a surjection onto the set of  positive elements of $S_{\cc}$.
\end{lemma}	
\begin{proof}
If $0\leq \overline{x}-\overline{e}$ then $e+f\leq x+f$ for some $f\in S_{\co}$.  By O5, $e+f$ is complemented in $x+f$. That is, there exists $x'\in S$ such that $x+f=x'+e+f$. Then $\overline{x'}=\overline{x}-\overline{e}$.	
\end{proof}

\begin{lemma}\label{lemmasups}
If $\ov x\leq \ov y$ and $x'\ll x$ then there exists $z\in S$ such that $x'\ll z$ and $\ov z=\ov y$.	
\end{lemma}	
\begin{proof}
Choose $x''$ such that $x'\ll x''\ll x$. Since $\ov x\leq \ov y$, there exists $e\in S_{\co}$ such that $x''+e\leq y+e$. By O5, $e$ is complemented in $y+e$ and this complement may be chosen way above $x'$. That is, there exists $z$ such that $z+e=y+e$ and $x'\ll z$. Thus, $z$ is as desired.	
\end{proof}

\begin{theorem}\label{suprema}
Every increasing sequence in  $S_{\cc}$ has a supremum. Moreover, if $(x_n)_{n=1}^\infty$ is an increasing sequence	of positive elements in $S_{\cc}$ then we can choose an increasing sequence $(z_n)_{n=1}^\infty$ in $S$  such that 	$\ov{z_n}\leq x_n$ for all $n$  and $\sup_n x_n=\ov {\sup_n z_n}$.
\end{theorem}

\begin{proof}
Let $(x_n)_{n=1}^\infty$ be an increasing sequence in $S_{\cc}$.
Given $e\in S_{\co}$, the translation map  $x\mapsto x+\ov e$ is an ordered set isomorphism of $S_{\cc}$. Thus,  applying one such translation
we may assume that $x_1\geq 0$. Then, by Lemma \ref{ontopositives}, for each $n$ we can choose $y_n\in S$ such that $x_n=\ov y_n$.

We construct recursively sequences $(z_n)_n$ and $(z_n')_n$ in $S$ satisfying that 
\begin{enumerate}[(1)]
\item
$\ov z_n'=\ov y_n$ for all $n$,
\item
$z_{n}\ll z_{n+1}\ll z_{n}'$ for all $n$, 

\item
if $y\ll y_m$ for some $m$ then $\ov y\leq \ov z_n$ for some $n$.
\end{enumerate}
Let us first choose for each $n$ a $\ll$-increasing sequence $(y_{n,k})_k\subseteq S$ such that $y_n=\sup_k y_{n,k}$. 
Set $z_1=0$ and  $z_1'=y_1$. 	
Assume that $z_1,z_1',\ldots,z_{n},z_{n}'$ have already been chosen. Since $\ov {y_{n-1}}\leq \ov y_n=\ov {z_n'}$, there exists $e\in S_c$ such that $y_{n-1,n+1}+e\leq z_n'+e$. We can thus choose $z_{n+1}\ll z_{n}'$ such that $y_{n-1,n}+e\leq z_{n+1}+e$. Hence, $\ov {y_{n-1,n}}\leq \ov{z_{n+1}}$. Moreover, increasing $z_{n+1}$ along an $\ll$-increasing sequence with supremum $z_n'$,  we can arrange that $\ov {y_{k,n}}\leq \ov{z_{n+1}}$ for $k=1,\ldots,n-1$ and also that $z_n\ll z_{n+1}$ (since $z_n\ll z_n'$). 
Now, using  that $\ov {z_n'}=\ov y_n\leq \ov {y}_{n+1}$ and Lemma \ref{lemmasups}, we choose $z_{n+1}'\in S$ such that $z_{n+1}\ll z_{n+1}'$ and $\ov {z_{n+1}'}=\ov y_{n+1}$.  We continue this process ad infinitum.
The sequences $(z_n)_n$
and $(z_n')_n$ obtained in this way satisfy (1), (2), and (3), above.

Observe that $(z_n)_n$ is increasing.
Let $z=\sup z_n$. Let us show that $\ov z=\sup \ov y_n$. Fix $y_m$. Let $y\ll y_m$. Then
$\ov y\leq \ov{z_{n}}$ for some $n$ ans so $\ov y\leq \ov{z}$.
Since this holds for all $y\ll y_m$, we conclude that $\ov{y_m}\leq \ov{z}$.
Thus, $\ov z$ is an upper  bound for the sequence $(\ov y_n)_n$. Suppose that $w\in S$ is such that
$\ov y_n\leq \ov w$ for all $n$. Choose $z'\ll z$. Then $z'\ll z_n\leq z_n'\sim_{\cc}y_n$ for some $n$.  
Hence, $z'\precsim_{\cc} w$. Since $z'\ll z$ is arbitrary, it follows that $z\precsim_{\cc} w$, as desired.
\end{proof}

\begin{lemma}\label{ovCumap}
The map $x\mapsto \overline{x}$ preserves suprema of increasing sequences and the compact containment relation.
\end{lemma}	

\begin{proof}
Let $(x_n)_n$ be an increasing sequence in $S$	 with  $x=\sup x_n$. Choose  $z\in S$ such that $\ov z=\sup \ov x_n$. Since $\ov{x_n}\leq \ov{x}$ for all $n$, we have that $\ov{z}\leq \ov{x}$. Let us prove the opposite inequality. 
Let $x'\ll x$. Then $x'\ll x_n$ for some $n$.  Since $\ov{x_n}\leq \ov{z}$, there exists $e\in S_c$ such that $x'+e\leq z+e$. This shows that $\ov{x}\leq \ov{z}$, as desired.

Let us now prove preservation of compact containment.	
Let $x', x\in S$ with $x'\ll x$. Suppose that $\ov{x}\leq \sup \ov{y_n}$, where $(\ov {y_n})_n$
is increasing.  By Theorem \ref{suprema}, there exists an increasing sequence
$(z_n)_n$ in $S$  such that $\ov{\sup z_n}=\sup \ov{y_n}$ and $\overline{z_n}\leq \overline{y_n}$ for all $n$. Set $z=\sup z_n$.
Since $\ov x\leq \ov z$,  there exist $z'\ll z$ and $e\in S_c$ such that $x'+e\leq z'+e$.
Then $z'\leq z_{n_0}$ for some $n_0$.
Hence, $\overline{x'}\leq \overline{z_{n_0}}\leq \overline{y_{n_0}}$.   This shows that $\ov {x'}\ll\ov{x}$.
\end{proof}

\begin{theorem}\label{SccCu}
Let $S$ be a positively ordered Cu-semigroup satisfying O5. 	Then the ordered semigroup $S_{\cc}$ defined above is a $\Cu$-semigroup also satisfying O5.
\end{theorem}

\begin{proof}
We have already proven the existence of sequential suprema in $S_{\cc}$.	

Let us prove O2.	
Given $\overline{x}-\ov{e}\in S_{\cc}$, we can choose  an $\ll$-increasing sequence $(x_n)_n$ in $S$  with supremum $x$.
Since,  by the previous lemma,  the map	$z\mapsto \ov z$ is  supremum and $\ll$ preserving, the sequence $(\ov x_n))$ is $\ll$-increasing and has supremum $\ov x$.
Finally, the map $\ov z\mapsto \ov z -\ov e$ is an order isomorphism. Thus, $(\ov{x_n}-\ov{e})$ is $\ll$-increasing and has supremum $\ov{x}-\ov{e}$.

Let us prove O3.	Translating the two sequences in this axiom by a suitable $\ov e$, with $e\in S_{\co}$,
we nay assume that their terms are positive.  Let 
$(\ov  x_n)$ and $(\ov  y_n)$ be increasing sequences of positive elements in $S_{\cc}$. One inequality is clear:
since $\ov x_n+\ov y_n\leq \sup \ov x_n+\sup \ov y_n$ for all $n$,
\[
\sup (\ov x_n+\ov y_n)\leq \sup \ov x_n+\sup \ov y_n.
\] 
Using Theorem \ref{suprema}, choose  an increasing sequences $(z_n)_n$ in $S$
such that $\ov z_n\leq \ov x_n$ and $\ov{\sup z_n}=\sup \ov x_n$. Similarly, choose
$(w_n)$ increasing and such that  and $\ov w_n\leq \ov y_n$ and $\ov{\sup w_n}=\sup \ov y_n$. 
By O3 in $S$ we have that
\[
\sup z_n+\sup w_n=\sup (z_n+w_n).
\]
Passing to $S_{\cc}$ and applying that $z\mapsto \ov{z}$ is supremum preserving on the left side we get
\[
\sup_n \ov x_n+\sup_n \ov y_n=\sup_n (\ov{z_n}+{z_n})\leq \sup_n (\ov x_n+\ov y_n). 
\]

Let us prove O4.	Translating by a suitable $\ov e\in S_{\co}$, we may again assume that all the elements involved 
are positive. Suppose that $\ov x_1\ll \ov x_2$ and $\ov y_1\ll \ov y_2$. Then $\ov x_1\leq \ov x_2'$ for some  $x_2'\ll x_2$
and $\ov y_1\ll \ov y_2'$ for some  $y_2'\ll y_2$. By O4 in $S$, $x_2'+y_2'\ll x_2+y_2$. Thus, using that $z\mapsto \ov z$ preserves
compact containment,
\[
\ov x_1+\ov y_1\leq \ov x_2'+\ov y_2'\ll \ov x_2+\ov y_2. 
\]

Finally, let us prove O5. Translating  by a suitable $\ov e\in S_{\co}$, we may again assume that all the elements involved  are positive. Suppose that we have elements $\ov x,\ov y,\ov w\in S_{\cc}$ such that $\ov x+\ov w\leq \ov y$. Let $x'\ll x''\ll x$, $w'\ll w''\ll w$. 
Then 
\[
x''+w''+e\leq y+e
\]
for some $e\in S_{\co}$. By O5 in $S$, there exists $z$ such that $w'\ll z$ and
\[
x'+z+e\leq y+e \leq x''+e+z.
\]
Passing to $S_{\cc}$ we find that $\ov x' +\ov z\leq \ov y\leq \ov x+\ov z$ and $\ov w'\ll \ov z$. This proves O5 in $S_{\cc}$.
\end{proof}

Let $\alpha\colon S\to T$ be a morphism of Cu-semigroups. Since $\alpha$ maps compact elements
to compact elements, we have a map $(x,e)\mapsto (\alpha(x),\alpha(e))$ from $S\times S_{\co}$ to $T\times T_{\co}$. 

\begin{lemma}
The map $(x,e)\mapsto (\alpha(x),\alpha(e))$ preserves the $\precsim_{\cc}$ relation.
\end{lemma}	

\begin{proof}
Let $(x,e)$ and $(y,f)$ be pairs in $S\times S_{\co}$ such that
$(x,e)\precsim_{\cc}(y,f)$. Let $z\ll \alpha(x)$. Choose $x'\ll x$ such that
$z\leq \alpha(x')\ll \alpha(x)$. Then $x'+f+g\leq y+e+g$ for some  $g\in S_{\co}$.
Then,
\[
z+\alpha(f)+\alpha(g)\leq \alpha(x')+\alpha(f)+\alpha(g)\leq \alpha(y)+\alpha(e)+\alpha(g).
\]
It follows that $(\alpha(x),\alpha(e))\precsim_{\cc} (\alpha(y),\alpha(f))$.
\end{proof}

In view of the previous lemma, we can  define
$\alpha_{\cc}\colon S_{\cc}\to T_{\cc}$ by
\[
\alpha_{\cc}(\ov x-\ov e)=\ov{\alpha(x)}-\ov{\alpha(e)}.
\]

\begin{theorem}\label{ccmaps}
The map  $\alpha_{\cc}\colon S_{\cc}\to T_{\cc}$ is a morphism of $\Cu$-semigroups. 
\end{theorem}
\begin{proof}
Additivity, preservation of order, and preservation of 0  are straightforward to check.	

Let us prove the preservation of sequential suprema.
Translating by compact elements,  it suffices to consider sequences of positive elements.	
Let $(\ov x_n)_n$ be an  increasing sequence in $S_{\cc}$.  
The inequality $\sup \alpha_{\cc}(\ov{x_n})\leq \alpha_{\cc}(\sup \ov {x_n})$ is clear.
Choose an increasing sequence $(z_n)_n\subseteq S$ such that $\ov{\sup z_i}=\sup \ov{x_n}$ and $\ov z_n\leq \ov x_n$ for all $n$. Set $z=\sup z_n$.
Since $\alpha$ is a morphism of Cu-semigroups,  $\sup_n \alpha(z_n)=\alpha(z)$. Passing to $S_{\cc}$  we get
\[
\alpha_{\cc}(\ov z)=\sup \alpha_{\cc}(\ov z_n)\leq \sup \alpha_{\cc}(\ov{x_n}),
\]
where we have used that the map $x\mapsto \ov x$ is supremum preserving. Thus, 
$\alpha_{\cc}(z)=\sup \alpha_{\cc}(\ov {x_n})$.

Finally, let us prove preservation of compact containment: Again translating all the elements involved it suffices to deal with positive elements  only.
Say $\ov x\ll \ov y$ in $S_{\cc}$. Since $\ov y=\sup_{y'\ll y} \ov y'$, there exists $y'\ll y$ such that  $\ov x\leq \ov y'$.  Now, $\alpha(y')\ll \alpha(y)$. 
Hence,
\[
\alpha_{\cc}(\ov x)\leq \alpha_{\cc}(\ov y')=\ov{\alpha(y')}\ll \ov{\alpha(y)}=\alpha(\ov y),
\]
where we have used that the map $z\mapsto \ov z$ preserves $\ll$.
\end{proof}

It is straightforward to check that $S\mapsto S_{\cc}$, $\alpha\mapsto \alpha_{\cc}$ is a functor from the category Cu to itself.

In \cite[Theorem 2]{Coward-Elliott-Ivanescu} it is shown that inductive limits exist in the category $\Cu$. It follows from the proof of this theorem that inductive limits in the category $\Cu$ can be characterized as follows: 	Let $(S_i, \alpha_{i,j})_{i,j\in I}$ be an inductive system in $\Cu$. A semigroup $S$ with maps $\alpha_{i,\infty}\colon S_i\to S$ is the inductive limit of 
$(S_i, \alpha_{i,j})_{i,j}$ in the category $\Cu$ if and only if the following two conditions are satisfied:
\begin{description}\label{indlimits}
\item[L1] For every $s\in S$ there exist $s_i\in S_i$ such that 
\[
\alpha_{i,i+1}(s_i)\le s_{i+1},\quad s=\sup \alpha_{i,\infty}(s_i).
\]
\item[L2] If $s,t\in S_i$, for some $i$, are such that $\alpha_{i,\infty}(s)\le \alpha_{i,\infty}(t)$ and $s'\ll s$ then there exists $j\geq i$ such that $\alpha_{i,j}(s')\le \alpha_{i,j}(t)$.
\end{description}

\begin{theorem}\label{SccContinuous}
The functor  $S\mapsto S_{\cc}$, $\alpha\mapsto \alpha_{\cc}$ preserves inductive limits.
\end{theorem}
\begin{proof}
Let $(S_i,\alpha_{i,j})_{i,j}$ be an inductive system with inductive limit $(S,\alpha_{i,\infty})_i$.
We must show that L1 and L2 above are satisfied after applying the $\cc$-construction. 

Let us prove L1, i.e., that every element in $S_{\cc}$ is the supremum of an increasing sequence of elements in $\bigcup_i \mathrm{im}((\alpha_{i,\infty})_{\cc})$.  Let $\ov x-\ov e\in S_{\cc}$. Every compact element in $S$ is the image of some compact element in some $S_i$. So there exists $e'\in S_i$ such that $\alpha_{i,\infty}(e')=e$. Passing to the inductive system $(S_j,\phi_{j,k})_{j,k\geq i}$, we can find an increasing sequence $(x_n)_n\subset  \bigcup_{j\geq i} \mathrm{im}(\alpha_{j,\infty})$  such that  $x=\sup x_n$. Then $\ov x-\ov e=\sup \ov x_n-\ov e$, and the sequence $(\ov x_n-\ov e)$ is contained in $\bigcup_i \mathrm{im}((\alpha_{i,\infty})_{\cc})$.

Let us  prove L2. Let $z,w\in (S_i)_{\cc}$  be such that $(\alpha_{i,\infty})_{\cc}(z)\leq (\alpha_{i,\infty})_{\cc}(w)$.  Say $z=\ov x-\ov e$ and $w=\ov y-\ov e$, where $x,y,e\in S_i$ and $e$ is a compact element. If $z'\ll z$ then $z'\leq \ov x'-\ov e$ for some $x'\ll x$. So it suffices to assume that $z'=\ov x'-\ov e$. Since $S$ is the inductive limit of the system $(S_j,\alpha_{j,k})$ in the category Cu, there exists $j\geq i$
such that $\alpha_{i,j}(x')\leq \alpha(y)$. It is then clear that $(\alpha_{i,j})_{\cc}(\ov x'-\ov e)\leq (\alpha_{i,j})_{\cc}(\ov y-\ov e)$, as desired.
\end{proof}

\section{The augmented Cuntz semigroup} \label{Cusim}
\subsection{Definition and basic properties}
Let $A$ be a C*-algebra. We now turn to the definition of the ordered semigroup  $\Cusim(A)$, which we call
the augmented Cuntz semigroup of the C*-algebra $A$.

Since $\Cu(A)$ is a Cu-semigroup satisfying O5, we can apply the $\cc$-functor from the previous section to it thereby obtaining a Cu-semigroup $(\Cu(A))_{\cc}$. We shall denote by $\Cu_{\cc}$ the functor obtained by first applying the Cuntz semigroup functor followed by the $\cc$-functor. We thus write $\Cu_{\cc}(A)$ and $\Cu_{\cc}(\phi)$, rather than $(\Cu(A))_{\cc}$ and $(\Cu(\phi))_{\cc}$, although these mean the same thing.

Let $\tilde A$ denote the forced unitization of $A$. Observe that $[1]\in \Cu(\tilde A)$ is a full compact element: for any compact  $e\in \Cu(\tilde A)$ there exists $n\in \N$ such that $e\leq n[1]$. It follows that
\[
\Cu_{\cc}(\tilde A)=\{\ov{[a]}-n\ov{[1]}:[a]\in \Cu(\tilde A),\, n\in \N\}.
\]
Let $\pi\colon \tilde A\to \C$ denote the canonical quotient map. The 
$\Cu$-morphism $\Cu(\pi)$ from $\Cu(\tilde A)$ to $\Cu(\C)\cong \overline\N$ gives rise to a $\Cu$-morphism
$\Cu_{\cc}(\pi)$ from $\Cu(\tilde A)$ to $ \Cu_{\cc}(\C)\cong \overline{\Z}$.
We call $\Cu_{\cc}(\pi)$ the rank map on $\Cu_{\cc}(\tilde A)$ and we alternatively denote it by $\rank(\cdot)$. 
Notice that under the identification
of $\Cu(\C)$ with $\overline{\N}$,  $\Cu(\pi)$
assigns to a  Cuntz class   $[a]\in \Cu(\tilde A)$, where $a\in (\tilde A\otimes \mathcal K)_+$, the rank
of  $\pi(a)\in \mathcal K$. We also denote this number by  $\rank(a)$, $\rank([a])$, or $\rank(\ov{[a]})$.

We define $\Cusim(A)$ as the kernel of the rank map on $\Cu_{\cc}(\tilde A)$, i.e.,
\[
\Cusim(A):=\Big\{\ov{[a]}-n\ov{[1]} : [a]\in \Cu(\tilde A),\,n\in\N\hbox{ such that }\rank(\pi(a))=n\Big\}.
\] 
We endow $\Cusim(A)$ with the addition and the order from $\Cu_{\cc}(\tilde A)$. 

Let  $\iota \colon A\to \tilde A$ denote the inclusion of $A$ in $\tilde A$.
This map induces an embedding of $\Cu(A)$ as a subsemigroup (indeed, as an ideal) of $\Cu(\tilde A)$. Let us regard $\Cu(A)$ as a subsemigroup of $\Cu(\tilde A)$ via this embedding.
Observe then that the rank map on $\Cu(\tilde A)$  is zero on   $ \Cu(A)$.
It follows that the map $Cu_{\cc}(A)\stackrel{\Cu_{\cc}(\iota)}{\to} Cu_{\cc}(\tilde A)$, induced by the inclusion of  $\Cu(A)$  in $\Cu(\tilde A)$ ranges in $\Cusim(A)$.
We thus get a map $\Cu_{\cc}(\iota)\colon \Cu_{\cc}(A)\to \Cusim(A)$,
\[
\Cu_{\cc}(A)\ni \ov x-\ov e\longmapsto \ov x-\ov e\in \Cusim(A).
\]  
\begin{proposition}\label{unitalCusim}
Let $A$ be a unital C*-algebra. Then the map  $\Cu_{\cc}(\iota)\colon \Cu_{\cc}(A)\to \Cusim(A)$ described above
is an isomorphism of ordered semigroups.
\end{proposition}	
\begin{proof}
Since $A$ is unital, $\tilde A\cong A\oplus \C$ via the homomorphism $a+\lambda \tilde 1\mapsto(a+\lambda 1,\lambda)$
(where $\tilde 1$ is the unit in $\tilde A$ and $1$ the unit in $A$).  Let us identify $\tilde A$ with $A\oplus \C$ via this isomorphism.
Observe then that $A$, as a subalgebra of $\tilde A$, is $A\times \{0\}$, and that the quotient homomorphism 
$\tilde A\to \C$ is  $(a,\lambda)\mapsto \lambda$.

We have $\Cu(\tilde A)=\Cu(A)\times \overline{\N}$ (on the right side order and addition are taken coordinatewise). 
The set of compact elements of  $\Cu(A)\times \overline{\N}$ is $\Cu_{\co}(A)\times \N$. It is straightforward to calculate that
$\Cu_{\cc}(\tilde A)\cong \Cu_{\cc}(A)\times \overline{\Z}$.  Further, the
rank map $\rank\colon \Cu_{\cc}(A)\to \overline{\Z}$ is simply $(\ov{[a]}-n\ov{[1]},m)\mapsto m$. In particular, the subsemigroup of rank zero elements $\Cusim(A)$ is $(\Cu_{\cc}(A),0)$. This is precisely the range of  the 
map $\Cu_{\cc}(\iota)\colon \Cu_{\cc}(A)\to \Cusim(A)$ described above. Indeed, this map is $\ov x-\ov e\mapsto (\ov x-\ov e,0)$. 
\end{proof}	
\begin{example}
Recall that $\Cu(\C)\cong \ov \N$. Thus, the preceding proposition  and Example \ref{Ncc} imply that  $\Cusim(\C)\cong \ov \Z$.
\end{example}

As remarked above, if $x\in \Cu(A)$ then $\ov x$  in $\Cu_{\cc}(\tilde A)$ has rank 0, and thus belongs to $\Cusim(A)$. Let us define $q\colon \Cu(A)\to \Cusim(A)$ as $q(x)=\ov x$.
\begin{lemma}\label{ontoq}
The map $q$ is a surjection onto the positive elements of $\Cusim(A)$.	
\end{lemma}	
\begin{proof}
Let $z\in \Cusim(A)\subseteq \Cu_{\cc}(\tilde A)$ be a positive element.  By Lemma \ref{ontopositives},   $z=\ov{x}$ for some $x\in \Cu(\tilde A)$.
But $\rank(x)=0$, so $x\in \Cu(A)$. Thus,   $q(x)=z$.	
\end{proof}

\begin{theorem}\label{CusimCu}
$\Cusim(A)$ is a Cu-semigroup. Moreover, it satisfies  the following  additional  axioms:
\begin{enumerate}[(i)]
\item  For all $x\in \Cusim(A)$ there exists $z\geq 0$ such that $x+z\geq 0$.
\item For all $x'\ll x$ there exists $z$ such that $x'+z\leq 0\leq x+z$.
Moreover, if $x+y\leq 0$ for some $y$ and $y'\ll y$ then $z$  can be chosen such that $y'\leq z$.
\end{enumerate}
\end{theorem}	
\begin{proof}
In Theorem \ref{SccCu}	we have  proved these properties for $\Cu_{\cc}(\tilde A)$. We show here that they pass on to $\Cusim(A)$. This is essentially a consequence of the rank map being a Cu-morphism. 

Let $(z_n)_n$ be an  increasing sequence in $\Cusim(A)$. 
Since $\rank(\cdot )$ is supremum preserving, the supremum of this sequence in $\Cu_{\cc}(\tilde A)$ has rank zero, i.e., it belongs to $\Cusim(A)$. 
Hence $\Cusim(A)$ is closed under sequential suprema and further  the supremum of a sequence in $\Cusim(A)$ agrees with the supremum
of the same sequence taken in $\Cu_{\cc}(\tilde A)$.   It is now clear that O3 holds in $\Cusim(A)$, given that it holds $\Cu_{\cc}(\tilde A)$

Let $z\in \Cusim(A)$. Choose  an $\ll$-increasing sequence $(z_n)_n$ in $\Cu_{\cc}(\tilde A)$ with supremum $z$.
Then $\sup \rank z_n=0$. Since $0$ is a compact element of $\ov\Z$, $\rank z_n=0$ for large enough $n$, i.e., $z_n\in \Cusim(A)$
for large enough $n$. Since the relation $\ll$ taken in $\Cu_{\cc}(\tilde A)$ is stronger than the same relation in $\Cusim(A)$, the sequence is also $\ll$-increasing in $\Cusim(A)$. This proves O2.

Let us show that the restriction of the relation $\ll$ in $\Cu_{\cc}(\tilde A)$ to $\Cusim(A)$ agrees with the same relation in $\Cusim(A)$. As already pointed out, one direction is clear.
Suppose that $x,y\in \Cusim(A)$ are such that $x\ll y$ in $\Cusim(A)$. Choose $(y_n)_n\subset \Cu_{\cc}(\tilde A)$ that is $\ll$-increasing and with  supremum $y$.
As observed in the previous paragraph, $y_n\in \Cusim(A)$ for large enough $n$. Thus, there exists $n_0$ such that $x\leq y_{n_0}\ll y$, where $\ll$ is taken in $\Cu_{\cc}(\tilde A)$.  Thus, $x\ll y$ in $\Cu_{\cc}(\tilde A)$. Axiom O4 in $\Cusim(A)$ is now immediate
from its holding in $\Cu_{\cc}(\tilde A)$. 

Let us prove property (i). Let $x=\ov{[a]}-n\ov{[1]}\in \Cusim(A)$, where $\rank(a)=n$. By the properties of the Cuntz semigroup functor under C*-algebra quotients (\cite{Ciuperca-Robert-Santiago}), there exist $[b],[c]\in \Cu(A)$ such that
\[
n[1]+[c] = [a]+[b],\hbox{ in }\Cu(\tilde A).
\] 
Let $z=\ov{[b]}\in \Cusim(A)$, i.e., $z$ is the image of $[b]$ under the map $q$ defined above.
Then, working in $\Cu_{\cc}(\tilde A)$, we have
\[
x+z=\ov{[a]}-n\ov{[1]}+\ov{[b]}=\ov{[a]+[b]}-n\ov{1}=\ov{[c]}\geq 0.
\] 

Finally, let us prove  (ii). Let $e=\ov{[1]}\in \Cu_{\cc}(\tilde A)$. Let $x', x\in \Cusim(A)$ be such that  $x'\ll x$. Choose $x''$  such that
$x'\ll x''\ll x$.  Since $\ov{[1]}$ is a full element  in $\Cu_{\cc}(\tilde A)$, there exists $N\in \N$ such that $x''\leq N\ov{[1]}$.
By  O5 in $\Cu_{\cc}(\tilde A)$, there exists $w\in \Cu_{\cc}(\tilde A)$ such that $x'+w\leq N\ov{[1]}\leq x''+w$.
Set $z=w-N\ov{[1]}$.
Then 
\[
x'+z\leq 0\leq x''+z\leq x+z.
\]
From the fact that $x'$ and $x$ both have rank 0 we deduce that  $z$ has rank 0 as well, i.e.,
it belongs to $\Cusim(A)$. Further, if $y',y$ are such that $x+y\leq 0$ and $y'\ll y$, then 
\[
x''+(y+N\ov{[1]})\leq N\ov{[1]}.
\] 
By O5 in $\Cu_{\cc}(\tilde A)$, $w$ may be chosen such that $y'+N\ov{[1]}\le w$, i.e., $y'\leq z$. So $z$ is as desired.
\end{proof}

\begin{remark}
If there is a greatest element in $\Cusim(A)$ (e.g., $A$ is separable),  axiom (i) of the previous theorem boils down to saying that $x+\omega=\omega$ for all $x\in \Cusim(A)$.
\end{remark}

Weak cancellation in a Cu-semigroup is defined as follows: $x+z\ll y+z$ implies $x\ll y$ for all $x,y,z$. The Cuntz semigroup of a C*-algebra of stable rank one has weak cancellation (\cite[Theorem 4.3]{Rordam-Winter}). This property, however,  is not true in general for the Cuntz semigroup of a C*-algebra.  On the other hand, $\Cusim(A)$ always has weak cancellation:

\begin{corollary}\label{weakcancellation}
If $x+z\ll y+z$ in $\Cusim(A)$ then $x\leq y$.	
\end{corollary}	
\begin{proof}
Let $z'\ll z$ be such that $x+z\leq y+z'$. Choose $w\in \Cusim(A)$ 
such that $z'+w\leq 0\leq z+w$. Then 
\[
x\leq x+z+w\leq y+z'+w\leq y.\qedhere
\]	
\end{proof}

Let $\phi\colon A\to B$ be a homomorphism of C*-algebras.
Let $\tilde\phi\colon \tilde A \to \tilde B$ denote the homomorphism obtained by applying the forced unitization functor.
This homomorphism gives rise to 
a Cu-morphism $\Cu(\tilde\phi)\colon \Cu(\tilde A)\to \Cu(\tilde B)$. 
Then, as shown in the previous section, we obtain 
a Cu-semigroup morphism  $\Cu(\tilde\phi)_{\cc}\colon \Cu_{\cc}(\tilde A)\to \Cu_{\cc}(\tilde B)$. From the commutativity
of the diagram
\[
\xymatrix{
\tilde A\ar[r]^{\tilde \phi}\ar[d]^{\pi}&\tilde B\ar[d]^{\pi}\\
\C\ar[r]^{\mathrm{id}}&\C
}
\]
we deduce that $\Cu_{\cc}(\tilde\phi)$ preserves the rank. In particular, it maps $\Cusim(A)$ to $\Cusim(B)$.
We  define $\Cusim(\phi)$ as the restriction of $\Cu(\tilde\phi)_{\cc}$ to $\Cusim(A)$. More concretely,
\[
\Cusim(\phi)(\ov{[a]}-n\ov{[1]})=\ov{[\tilde\phi\otimes\mathrm{id}(a)]}-n\ov{[1]}.
\] 
(Here $\mathrm{id}$ is the identity on $\mathcal K$, so that $\tilde\phi\otimes \mathrm{id}\colon \tilde A\otimes\mathcal K\to \tilde B\otimes \mathcal K$.)

\begin{theorem}\label{Cusim_morphisms}
$\Cusim(\phi)$ is a Cu-morphism.
\end{theorem}	
\begin{proof}
In the course of the prove of Theorem \ref{CusimCu} we have shown that both the suprema of increasing sequences and the compact containment
relation in $\Cusim(A)$ agree with their counterparts taken in the larger ordered semigroup $\Cu_{\cc}(\tilde A)$. It is then clear that the preservation
of sequential suprema and of compact containment by $\Cusim(\phi)$ follow from the same properties for the map $\Cu_{\cc}(\tilde \phi)$. But, indeed,  $\Cu_{\cc}(\tilde \phi)$ is a Cu-morphism by Theorem \ref{ccmaps}. 
\end{proof}

\subsection{Comparison with the original definition}\label{cforiginal}
Let us  compare the definition of $\Cusim(A)$ given above with the original definition in \cite{Robert}. The main difference lies in the definition of the relation $\precsim_{\cc}$, which in \cite{Robert} is taken as $(x,e)\precsim_{\cc}(y,f)$
if 
\[
x+f+g\leq y+e+g
\]
for some compact element $g$. The rest of the construction is the same. Observe that the relation used in \cite{Robert} is a stronger relation, i.e., if $(x,e)\precsim_{\cc}(y,f)$ as in \cite{Robert} then $(x,e)\precsim_{\cc}(y,f)$ as defined above. If these two relations are the same in $\Cu(\tilde A)\times \Cu(\tilde A)_{\co}$, then the two constructions of $\Cusim(A)$ agree. In Corollary \ref{srm} below we show that this is the case for C*-algebras of finite stable rank. Thus, in this case the two constructions agree.

Let us recall the definition of the stable rank of a C*-algebra. Let $m\in \N$ and let $A$ be unital C*-algebra. An $m$-tuple $(a_1,\ldots,a_m)\in A^m$ is called left invertible if 
$\sum_{i=1}^m a_ib_i=1$ for some $(b_1,\ldots,b_m)\in A^m$. The C*-algebra
$A$ is said to have stable rank at most $m$ if the set of left invertible $m$-tuples
is dense in $A^m$.  This can be rephrased in the language of Hilbert C*-modules as follows: 
$A$ has stable rank at most $m$ if the left invertible operators from $A$ to $A^m$ are dense in the set of bounded operators from $A$ to $A^m$. The stable rank of $A$ is the least $m$ with this property. If now such $m$ exists then the stable rank is infinite.  For  non-unital C*-algebras the stable rank is defined as the stable rank of the unitization.

\begin{theorem}
Let $A$ be a unital C*-algebra of stable rank at most $m<\infty$. Let $H$ and $G$ be 
Hilbert C*-modules over $A$ such that $A^{m+1}\oplus H\cong A\oplus G$. 
Then $A^m\oplus H\cong G$.
\end{theorem}	
\begin{proof}
Let $H'=A^m\oplus H$, so that $A\oplus H'\cong A\oplus G$. 
Let $V\colon A\oplus G\to A\oplus H'$ be a  Hilbert C*-module isomorphism. Let us write $V$ in matrix form:
\[
\begin{pmatrix}
V_{11} & V_{12}\\
V_{21} & V_{22} 
\end{pmatrix}\in 
\begin{pmatrix}
B(A) & B(G,A)\\
B(A,H') & B(G,H') 
\end{pmatrix}.
\]
Recall that $H'=A^m\oplus H$. Let us write $V_{21}=(U_1,U_2)$, where $U_1\in B(A,A^m)$ and $U_2\in B(A,H)$. Let $U_1'\in B(A,A^m)$ be left invertible and $\|U_1-U_1'\|<1$.
Let $V_{21}'=(U_1',U_2)$.  Now  $V_{21}'$ is left invertible and $\|V_{21}'-V_{21}\|<1$. If we replace $V_{21}$ by $V_{21}'$ in the matrix
of $V$ we get $V'$ such that $\|V-V'\|<1$. Thus $V'$ is an invertible adjointable operator. Let us assume instead that $V$ is an invertible
adjointable operator such that $V_{21}\in B(A,H')$ is left invertible. Let $W\in B(A,H')$ be a left inverse of $V_{21}$. Set $(1-V_{11})W=X$. Then
\[
\begin{pmatrix}
1 & X\\
0 & 1 
\end{pmatrix}
\begin{pmatrix}
V_{11} & V_{12}\\
V_{21} & V_{22} 
\end{pmatrix}=
\begin{pmatrix}
1 & V_{12}'\\
V_{21} & V_{22} 
\end{pmatrix},
\]
and 
\[
\begin{pmatrix}
1 &  0\\
-V_{21} & 1 
\end{pmatrix}
\begin{pmatrix}
1 & V_{12}'\\
V_{21} & V_{22} 
\end{pmatrix}
\begin{pmatrix}
1 & -V_{12}'\\
0& 1 
\end{pmatrix}=
\begin{pmatrix}
1 & 0\\
0 & V_{22}' 
\end{pmatrix}.
\]
The map on the right side is still an invertible adjointable operator.  It follows that $V_{22}'\in B(G,H)$ is adjointable and invertible. Hence $G\cong H$ as Hilbert C*-modules, via the polar decomposition of $V_{22}'$.
\end{proof}

\begin{theorem}
Let $A$ be a unital C*-algebra of stable rank at most $m<\infty$. If $x,y\in \Cu(A)$ are such that $x+(m+1)[1]\leq y+[1]$ then $x+m[1]\leq y$.
\end{theorem}		
\begin{proof}
We rely on the Hilbert C*-modules picture of $\Cu(A)$ (see \cite{Coward-Elliott-Ivanescu}). Suppose that $x=[H]$ and $y=[G]$, where $H$ and $G$ are countably generated Hilbert C*-modules over $A$. Let $E\subseteq H$ be a compactly contained submodule. The assumption $x+(m+1)[1]\leq y+[1]$ implies that $E\oplus A^{m+1}$ embeds as a Hilbert C*-submodule in $G\oplus A$. 
Let $\phi\colon E\oplus A^{m+1}\to G\oplus A$ be an embedding. Set $E'=\phi(0\oplus A^{m+1})^\perp$. That is, $E'$ is the orthogonal complement of the image of $A^{m+1}$ in $G\oplus A$. Since $A^{m+1}$ is a projective module, we have that $E'\oplus A^{m+1}\cong G\oplus A$.
By the previous theorem, $E'\oplus A^m\cong G$. Since $E$ embeds in $E'$ via $\phi$, $E\oplus A^m$ embeds in $G$. Hence $[E]+m[A]\leq [G]$. Passing to the supremum over all $E$ compactly contained in $H$ we get $[H]+m[A]\leq [G]$, as desired.
\end{proof}

\begin{corollary}\label{srm}
Let $A$ be a C*-algebra of  stable rank at most $m<\infty$. Let $[a],[b]\in \Cu(\tilde A)$. Then   $\ov{[a]} \leq \ov{[b]}$ in $\Cu_{\cc}(\tilde A)$ if and only if $[a]+m[1]\leq [b]+m[1]$ in $\Cu(\tilde A)$.	
\end{corollary}	

\begin{proof}
Let $[a']\ll [a]$. Then, by the definition of $\precsim_{\cc}$,  $[a']+n[1]\leq [b]+n[1]$ for some $n\in \N$. Since the stable rank of $\tilde A$ is at most $m$, by the previous theorem we have $[a']+m[1]\leq [b]+m[1]$. Passing to the supremum over all $[a']\ll [a]$, we get $[a]+m[1]\leq [b]+m[1]$.	
\end{proof}	

\subsection{Examples}
Let $X$ be a locally compact Hausdorff space and let $S$ be an ordered semigroup. Let  $\mathrm{Lsc}(X, S)$ denote the set of lower semicontinuous functions from $X$ to $S$, i.e., the set of $f\colon X\to S$ such that $\{x\in X\mid f(x)>s\}$ is open for every $s\in S$.  Let us regard $\mathrm{Lsc}(X, S)$ as an ordered semigroup endowed with the pointwise order and the pointwise addition operation. Let $\ov X=X\cup\{\infty\}$ denote the one-point compactification
of $X$. Let   $\mathrm{Lsc}_0(X, S)=\{f\in \mathrm{Lsc}(\ov X, S):f(\infty)=0\}$.

\begin{example}\label{CusimX} Let $X$ be a compact metrizable space of dimension at most 1. It was shown in \cite[Theorem]{Robert2} that the map
\begin{align*}
[f]\mapsto (x\in X \to \rank(f(x))),
\end{align*}
from $\Cu(C(X))$ to $\mathrm{Lsc}(X, \ov \N)$ is an order isomorphism. 

Notice that $C(X)$ has stable rank one. So we are free to use the simpler form of the relation $\sim_{\cc}$ used in \cite{Robert}. We thus get that $\Cu_{\cc}(C(X))$ is isomorphic to $\mathrm{Lsc}(X, \ov \Z)$ via the map 
\begin{align*}
[f]-[p]\mapsto \left(x\in X\to \rank(f(x))-\rank(p(x))\right), \quad \begin{aligned} &f,p\in \mathrm{C}(X,\K),\\& p \text{ a projection},\end{aligned}
\end{align*}
Further, since $C(X)$ is unital, 
\[
\Cusim(C(X))\cong \Cu_{\cc}(C(X))\cong \mathrm{Lsc}(X,\ov{\Z}\})
\] 
as ordered semigroups.
\end{example}

\begin{example}\label{CusimX0}	Let $X$ be a  locally compact, $\sigma$-compact, metrizable  space of dimension at most one. Let $\overline X=X\cup \{\infty\}$ denote its one-point compactification.  Then $\widetilde{C_0(X)}\cong C(\ov X)$. 
As argued in the previous example, $\Cu_{\cc}(\tilde X)\cong \mathrm{Lsc}(X, \ov{\Z}\})$. The rank map on  $\Cu_{\cc}(\tilde X)$, after its identification with $\mathrm{Lsc}(X, \ov\Z)$, is simply $\mathrm{Lsc}(X, \ov\Z)\ni g\mapsto g(\infty)\in \ov \Z$. Therefore, $\Cusim(C_0(X))\cong \mathrm{Lsc}_0(X, \ov\Z)$.
\end{example}

The following example illustrates how the positive and negative elements of $\Cusim(A)$ do not always span the entire $\Cusim(A)$
as a semigroup:

\begin{example} Here we calculate $\Cusim(C_0(\R^2))$. Let $S^2$ denote the 2-dimensional sphere.
Let us recall the computation of the Cuntz semigroup of $C(S^2)$  obtained in \cite{Robert2}.
Let $V(C(S^2))$ denote the Murray-von Neumann monoid of projections. Then 
\[
\Cu(C(S^2))\cong V(C(S^2))\sqcup \mathrm{Lsc}(S^2,\ov{\N})/\!\!\sim,
\] 
where $\sim$ is defined such that $n[1]\in V(S^2)$ and the constant function $n\in \mathrm{Lsc}(S^2,\ov{\N})$ are equivalent
for all $n=0,1,\ldots$. Addition on the right side is as follows: within the sets $V(C(S^2))$ and $\mathrm{Lsc}(S^2,\ov{\N})$
we simply use the addition operation with which these sets are endowed. If $[p]\in V(C(S^2))$ and $f\in \mathrm{Lsc}(S^2,\ov{\N})$
is non-constant then $[p]+f=\rank(p)+f\in \mathrm{Lsc}(S^2,\ov{\N})$. The order again need only be defined for 
$[p]\in V(C(S^2))$ and $f\in \mathrm{Lsc}(S^2,\ov{\N})$: $f\leq [p]$ if  $f\leq \rank(p) $  and $[p]\leq f$
if $\rank(p)\leq f$ and $f$ is non-constant. 

It is straightforward to calculate that 
\[
\Cu_{\cc}(C(S^2))\cong K_0(C(S^2))\sqcup \mathrm{Lsc}(S^2,\ov{\Z})/\!\!\sim,
\]
where $\sim$ identifies  $n[1]\in K_0(C(S^2))$ and the constant function $n\in \mathrm{Lsc}(S^2,\ov{\Z})$ for all $n\in \Z$.

Let us regard $C(S^2)$ as unitization of  $C_0(\R^2)$ and $S^2\cong \R^2\cup\{\infty\}$. 
The rank function on $\Cu_{\cc}(C(S^2))$ coming from this unitization  is then $\rank ([p]-[q])=\rank(p)-\rank(p)$ for $[p]-[q]\in K_0(C(S^2))$
and $\rank (f)=f(\infty)$ for $f\in \mathrm{Lsc}(S^2,\ov{\Z})$.  We thus get
\begin{align*}
\Cusim(C_0(\R^2)) &\cong K_0(\R^2)\sqcup \mathrm{Lsc}_0(\R^2,\ov{\Z})/\!\!\sim\\
&\cong \Z\sqcup \mathrm{Lsc}_0(\R^2,\ov{\Z})/\!\!\sim,
\end{align*}
where $\sim$ identifies $0\in \Z$ with $0\in \mathrm{Lsc}_0(\R^2,\ov{\Z})$.
Let us describe the order and addition on the set on the right.  Order: The elements of $\Z$ are pairwise not comparable. They are greater than any non-zero function 
in $\mathrm{Lsc}_0(\R^2,\ov{\Z})$ that is non-positive and  smaller than any non-zero function that is non-negative. The order on $\mathrm{Lsc}_0(\R^2,\ov{\Z})$ is pointwise. Addition:  if $n\in \Z$ and $f\in \mathrm{Lsc}_0(\R^2,\ov{\Z})$  then $n+f=f$. Observe that the subsemigroup
spanned by the positive and negative elements is $\mathrm{Lsc}_0(\R^2,\ov{\Z})$.
\end{example}

\section{Properties of the functor $\Cusim$}\label{Cusimfunctor}
By Theorems \ref{CusimCu} and \ref{Cusim_morphisms} of the  previous section, $A\mapsto \Cusim(A)$, $\phi\mapsto \Cusim(\phi)$, defines a functor from the category
of C*-algebras to the category Cu. Here we investigate its properties.
\subsection{Continuity with respect to inductive limits}
\begin{theorem}\label{Cusimlimits}
The functor $\Cusim$ is continuous, i.e, it preserves inductive limits.
\end{theorem}	
\begin{proof}
Let $(A_{i},\phi_{i,j})_{i,j\in I}$	 be an inductive system of C*-algebras and $(A, \phi_{i,\infty})_{i\in I}$ its  inductive limit. Both the  forced unitization functor and the Cuntz semigroup functor are continuous (\cite{Antoine-Perera-Thiel}). Thus, we have 
$\Cu(\tilde A)=\varinjlim \Cu(\tilde A_{i})$.  Moreover, by Theorem \ref{SccContinuous}, applying the ${\cc}$-construction still yields an inductive limit of Cu-semigroups:  $\Cu_{\cc}(\tilde A)=\varinjlim \Cu_{\cc}(\tilde A_{i})$

For brevity of notation, let us set $S=\Cu_{\cc}(\tilde A)$, $S_i=\Cu_{\cc}(\tilde A_i)$, and $\alpha_{i,j}=\Cu_{\cc}(\tilde \phi_{i,j})$. Each of the Cu-semigroups $S_i$ and the Cu-semigroup $S$ carry a rank map into $\ov \Z$.
As argued in the definition of the functor $\Cusim(\cdot)$ in the previous section, the Cu-morphisms $\alpha_{i,j}$ and $\alpha_{i,\infty}$ map rank zero elements to rank zero elements. So their restrictions to the rank zero elements form an inductive system as well. In order to prove that the Cu-subsemigroup of rank zero elements in $S$ is the inductive limit of the Cu-subsemigroups of rank zero elements in the $S_i$s, we must show that conditions  L1 and L2 of an inductive limit in the category $\Cu$ are valid (see Section \ref{ccconstruction}).  This is quite straightforward: Say $\ov x-\ov e\in S$ has rank zero. By L1 applied to the inductive limit $S=\varinjlim S_i$, we  can choose an increasing sequence $(\ov x_n-\ov e)_{n=1}^\infty\subset  \bigcup \mathrm{im}(\alpha_{i,\infty})$ with supremum $\ov x-\ov e$. The terms of this sequence eventually have rank zero (since $0\ll 0$ in $\ov \Z$). This proves L1. Condition  L2  follows also from the same condition in  the inductive limit $S=\varinjlim S_i$, bearing in mind that the compact containment relation on the the set of rank zero elements agrees with the compact containment relation on the larger set $S$ restricted to the rank zero elements (as argued in the proof of Theorem \ref{CusimCu}).
\end{proof}

\subsection{Split exactness}
A projection $p$ in a C*-algebra is called finite if it is not Murray-von Neumann equivalent to a proper subprojection of itself. In terms of the monoid of Murray-von Neumann classes, this is expressible as $[p]_{\mathrm{MvN}}+[q]_{\mathrm{MvN}}=[p]_{\mathrm{MvN}}$ implies $[q]_{\mathrm{MvN}}=0$ for all $[q]_{\mathrm{MvN}}\in V(A)$. A unital C*-algebra is called stably finite if any $[p]_{\mathrm{MvN}}\in V(A)$ is finite; equivalently, if $n[1]_{\mathrm{MvN}}$ is finite for all $n\in \N$.
 
\begin{lemma}\label{ovone}
Let $A$ be a stably finite unital C*-algebra and let $x\in \Cu(A)$. If $\ov{x}\in \Cu_{\cc}(A)$ is compact then $x=[p]$ for some projection $p\in A\otimes\mathcal K$.
\end{lemma}

\begin{proof}
We will make use of the characterization given in \cite{Brown-Ciuperca} of the  compact elements in the Cuntz semigroup of a stably finite C*-algebra. By \cite[Proposition 5.7]{Brown-Ciuperca},  if $A$ is unital and stably finite and $[a]\in \Cu(A)$ is compact, then $0$ is an isolated point of the spectrum of $a$. It follows, by functional calculus, that we can choose a projection $p$ such that $[a]=[p]$. It also follows that if $[a]+[b]$ is compact then both $[a]$ and $[b]$ are compact. Indeed,  once  0 is an isolated point of the spectrum  of $a\oplus b$, it is also an isolated point of the spectrum of $a$ and $b$. 

Let $(x_n)_n$ be rapidly increasing with supremum $x$. Since the map $z\mapsto \ov{z}$
preserves sequential suprema, $\ov{x_n}=\ov{x}$ for all sufficiently large $n$. Assume without loss of generality that this is the case for all $n$. Fix $n$. From $\ov{x_{n+1}}=\ov{x_1}$ we get  that  $x_n+k_n[1]\leq x_1+k_n[1]$ for some $k_n\in \N$. It follows that $x_n+k_n[1]=x_1+k_n[1]$ is compact, and so, as remarked above, $x_n=[p_n]$  for some projection $p_n$. Moreover, by stable finiteness, we must have that $[p_1]=[p_{n}]$. Indeed, if we write $[p_{1}]+[q]=[p_{n}]$, with $q$ a projection, then $[p_1]+k_n[1]+[q]=[p_1]+k_n[1]$, which implies that $[q]=0$. Thus, $[p_1]=[p_n]$ for all $n$, and so $x=[p_1]$.
\end{proof}

\begin{theorem}\label{splitexact}
Let
\[
\xymatrix{
0\ar[r]& I \ar[r]^\iota\  & A\ar[r]^\pi & A/I\ar[r] &0
}
\]
be a short exact sequence of C$^*$-algebras. Apply the functor $\Cusim$ to it to get
\[
\xymatrix{
\Cusim(I) \ar[r]^{\Cusim(\iota)}\  & \Cusim(A)\ar[r]^{\Cusim(\pi)} & \Cusim(A/I).
}
\]
The following are true:
\begin{enumerate}[(i)]
\item 
If $\tilde A/I$ is a stably finite C*-algebra then $\mathrm{Im}(\Cusim(\iota))=\mathrm{Ker}(\Cusim(\pi))$.

\item 
If $\pi$ splits then $\Cusim(\pi)$ is surjective and $\Cusim(\iota)$ is an order embedding. 
\end{enumerate}
\end{theorem}
\begin{proof}
(i) Let us first go over some notation: We denote by $\tilde\pi\colon \tilde A\to \widetilde{A/I}$ the unital extension of $\pi$.
Further, we also denote by $\tilde\pi$ the homomorphism $\tilde \pi\otimes \mathrm{id}$ from $\tilde A\otimes\mathcal K$
to $\widetilde{A/I}\otimes\mathcal K$. The same conventions apply to $\widetilde{\iota}$.
Finally, we identity $\widetilde{A/I}$ with $\widetilde{A}/I$ and  regard $\tilde I$ as a subalgebra of $\tilde A$.

From $\pi\circ\iota=0$ we get that $\Cusim(\pi)\circ\Cusim(\iota)=0$. Therefore,
$\mathrm{Im}(\Cusim(\iota))\subseteq \mathrm{Ker}(\Cusim(\pi))$. Next we show the opposite inclusion.

Let $\ov{[a]}-n\ov{[1]}$ be an element of $\Cusim(A)$ mapped to 0 by $\Cusim(\pi)$. 
That is, $\ov{[\tilde \pi(a)]}=n\ov{[1]}$ in $\Cu_{\cc}(\tilde A/I)$. Since $n\ov{[1]}$ is compact, we have by Lemma \ref{ovone} that $[\tilde\pi(a)]$ is a compact element of 
$\Cu(\tilde{A}/I)$. From the definition of $\precsim_\cc$, this means that  $[\tilde \pi(a)]+k[1]=n[1]+k[1]$ for some $k\in \N$. Let us replace $[a]$ by $[a]+k[1]$ and $n$ by $n+k$, which does not change $\ov{[a]}-n\ov{[1]}$, so as to now have $[\tilde \pi(a)]=n[1]$ in $\Cu(\tilde A/I)$. Since $\tilde \pi(a)$ is Cuntz equivalent to a projection and we have assumed that $\tilde A/I$ is stably finite, 0 is an isolated point of the spectrum of $\tilde \pi(a)$ (\cite[Proposition 5.7]{Brown-Ciuperca}). Hence, for a suitable choice of a strictly positive $f\in C_0(\R^+)$,  $\tilde\pi(f(a))$  is a projection. Let $a'=f(a)$. Then $[a']=[a]$ (since we have chosen $f$  strictly positive) and $\tilde\pi(a')$ is a projection. Let us simply rename $a'$ as $a$ and assume that $\tilde\pi(a)$ is a projection. Now  $\tilde \pi(a)$ and  $1_n$ are Cuntz equivalent projections in $(\widetilde{A}/I)\otimes\mathcal K$. Since $\widetilde{A}/I$ is stably finite, $\tilde \pi(a)$ and $1_n$ are Murray-von Neumann equivalent. In a stable C*-algebra, if $p,q$ are Murray-von Neumann equivalent projections then there exists a unitary $u$ in the unitization of the algebra  such that $upu^*=q$; moreover  $u$ may be chosen in the connected component of the identity. (Proof: If $\|p-q\|<1$ this is well known. From this we deduce the case when  $p$ and $q$ are homotopic. But in a stable C*-algebra Murray-von Neumann equivalent projections are homotopic.) Applying this fact  to $\tilde \pi(a)$ and $1_n$, we obtain a unitary  $u$ in the unitization of $(\widetilde A/I)\otimes \mathcal K$  such that $u\pi(a)u^*=1_n$ and   $u$ is in the connected component of the identity. Since $u$ is connected to the identity, it has a  lift  to a unitary $v$ in the unitization of $\tilde A\otimes\mathcal K$. Set $v^*av=a_1\in \tilde A\otimes \mathcal K$. Then $[a]=[a_1]$ and $\tilde \pi(a_1)=1_n$. The latter implies that $a_1\in \tilde I\otimes\mathcal K$. Thus, 
\[
\ov{[a]}-n\ov{[1]}=\ov{[a_1]}-n\ov{[1]}\in \mathrm{Im}(\Cusim(\iota)),
\] 
as desired.

(ii) Let $\lambda\colon A/I\to A$ be  a homomorphism such that $\pi\circ \lambda=\mathrm{id}_{A/I}$. Then $\Cusim(\psi)\circ\Cusim(\lambda)=\mathrm{id}_{\Cusim(A/I)}$, which implies that $\Cusim(\psi)$ is surjective.

Let us prove that $\Cusim(\phi)$ is an order embedding. We maintain the conventions introduced in the first paragraph of the proof of (i). Let $a,b\in \tilde I\otimes\mathcal K$ be positive elements such that $\ov{[a]}-n\ov{[1]}$ and $\ov{[b]}-n\ov{[1]}$ belong to $\Cusim(A)$
and $\ov{[a]}-n\ov{[1]}\leq \ov{[b]}-n\ov{[1]}$. Then $\rank(a)=\rank(b)=n$ (these are the ranks of their images in $\mathcal K$ by the canonical map $\tilde I\otimes\mathcal K\to \mathcal K$)
and $\ov{[a]}\leq \ov{[b]}$ in $\Cu_{\cc}(\tilde A)$. We will show that $\ov{[a]}\leq \ov{[b]}$
in $\Cu_{\cc}(\tilde I)$. Choose $[a']\ll [a]$. Then  there exists $k\in \N$ such that  $[a']+k[1]\leq [b]+k[1]$ in $\Cu(\tilde A)$. Suppose that we have shown that $[a']+k[1]\leq [b]+k[1]$ in $\Cu(\tilde I)$. Since $[a']\ll [a]$ has been chosen arbitrarily, we conclude that $\ov{[a]}\leq \ov{[b]}$ in $\Cu_{\cc}(\tilde I)$. It thus suffices to show that if $[a]\leq [b]$ in $\Cu(\tilde A)$ and $\rank(a)=\rank{b}=n$, then $[a]\leq [b]$ in $\Cu(\tilde I)$.
We prove this next.

We may assume without loss of generality that $a=1_n+a'$ and  $b=1_n+b'$, for some selfadjoint elements  $a',b'\in I\otimes \mathcal K$. Let $\epsilon>0$. Since $[a]\leq  [b]$ in $\Cu(\tilde A)$, there exists $x\in \tilde A\otimes\mathcal K$ such that $(a-\epsilon/2)_+=x^*x$ and $xx^*\leq Mb$ for some $M>0$. By the almost stable rank one property of $A\otimes \mathcal K$ (\cite[Lemma 4.3.2]{BRTTW}) there exists a unitary $u\in (A^\sim\otimes\mathcal K)^\sim$ such that  $u^*(a-\epsilon)_+u\leq Mb$.

Consider the elements
\begin{align*}
u_1 &=(\tilde \lambda\circ\tilde \pi)(u)\in (\tilde A\otimes\mathcal K)^\sim,\\
a_1 &=(uu_1^*)^*a(uu_1^*)\in \tilde I\otimes\mathcal K.
\end{align*}
We have $(\tilde\lambda \circ\tilde\pi)(uu_1^*)=1$, which implies that $uu_1^*$ is a unitary in the unitization of $\tilde I\otimes\mathcal K$. Thus, $a$ and $a_1$ are unitarily equivalent in $\tilde I\otimes \mathcal K$, which implies that  $[(a-\epsilon)_+]=[(a_1-\epsilon)_+]$ in 
$\Cu(\tilde I)$.  From $u^*(a-\epsilon)_+u\leq Mb$ and the definition of $u_1$  we get that $u_1^*(a_1-\epsilon)_+u_1\leq Mb$. Thus, on one hand we have that $[(a-\epsilon)_+]=[(a_1-\epsilon)_+]$ in $\Cu(\tilde I)$, and on the other hand we have that $u_1^*(a_1-\epsilon)_+u_1\leq Mb$. We  show next that $[(a_1-\epsilon)_+]=[u_1^*(a_1-\epsilon)_+u_1]$ in $\Cu(\tilde I)$. In fact, we will show that $a$ and $a_1$ are approximately unitarily equivalent in $\tilde I\otimes \mathcal K$ (with unitaries chosen in the unitization of this C*-algebra).

Applying $\tilde\lambda\circ\tilde\pi$ on both sides of $u_1^*(a_1-\epsilon)_+u_1\leq Mb$, and using that $(\tilde\lambda\circ\tilde\pi)(u_1)=u_1$,  we get that $u_1^*1_nu_1=1_n$, i.e., $u_1$ commutes with $1_n$. We also have that 
\[
(\tilde \lambda \circ\tilde \pi)(a_1)=(\tilde \lambda \circ\tilde\pi)(a)=1_n.
\] 
This implies that  $a_1=1_n+a_1'$ for some selfadjoint $a_1'\in I\otimes\mathcal K$. 
Let us choose an approximate unit $(e_\lambda)\in I_+$ of $I$. For each $\lambda$ and $k\in \N$ define
\[
v_{\lambda,k}=1_n+(u_1-1_n)(e_\lambda\otimes 1_k).
\]
Then $v_{\lambda,k}$ belongs to $\tilde I\otimes\mathcal K$ for all $\lambda$ and $k$. We have
\[
v_{\lambda,k}a_1'=1_na_1'+(u_1-1_n)(e_\lambda\otimes 1_k)a_1'\to u_1a_1',
\]
as $\lambda,k\to \infty$. Here we have used that $(e_{\lambda}\otimes 1_k)a_1'\to a_1'$ as $\lambda,k\to \infty$. Similarly, we deduce that $v_{\lambda,k}^*a_1'\to u_1^*a_1'$. Since $u_1-1_n$ and $e_\lambda\otimes 1_k$ commute with $1_n$, it is clear from its definition
that $v_{\lambda,k}$ commutes with $1_n$ for all $\lambda,k$.

Since $\tilde I\otimes\mathcal K$ has almost stable rank 1 (\cite[Lemma 4.3.2]{BRTTW}),  for each $\lambda,k$ there exists a unitary $w_{\lambda,k}$ in the unitization of $\tilde I\otimes\mathcal K$ such that $\|v_{\lambda,k}-w_{\lambda,k}|v_{\lambda,k}|\|<\frac{1}{k}$. 
Then 
\[
w_{\lambda,k}^*a_1' w_{\lambda,k}\to u_1^*a_1'u_1\hbox{ and  }w_{\lambda,k}^*1_n w_{\lambda,k}\to 1_n,\hbox{ as }\lambda,k\to \infty.\] 
It follows that $w_{\lambda,k}^*a_1w_{\lambda,k}\to u_1^*a_1u_1$. So $a$ and $a_1$ are approximately unitarily equivalent in $\tilde I\otimes \mathcal K$.  We now have 
\[
(a-\epsilon)_+\sim_{\Cu} (a_1-\epsilon)_+\sim_{\Cu} u_1^*(a_1-\epsilon)_+u_1\leq Mb,
\]
where the Cuntz comparisons are all taken in $\tilde I\otimes\mathcal K$.  Since $\epsilon>0$ is arbitrary, we conclude that $[a]\leq [b]$ in $\Cu(\tilde I)$. 
\end{proof}

\begin{corollary}\label{Cusimsums}
Let $A$ and $B$ be  C$^*$-algebras. Let $\iota_A\colon A\to A\oplus B$
and $\iota_B \colon B\to A\oplus B$ denote the standard inclusions. Then  
$\gamma\colon \Cusim(A)\oplus \Cusim(B)\to\Cusim(A\oplus B)$  given by
\[
\gamma(x+y):=\Cusim(\iota_A)(x)+\Cusim(\iota_B)(y)
\]
is an isomorphism of ordered semigroups.
\end{corollary}

\begin{proof}
Consider the diagram
\[
\xymatrix{
0\ar[r]& \Cusim(A) \ar[r]^-{\alpha}\ar@{=}[d]  & \Cusim(A)\oplus \Cusim(B)\ar[r]^-{\beta}\ar[d]^\gamma & \Cusim(B) \ar[r]\ar@{=}[d] &0\\
0\ar[r] & \Cusim(A) \ar[r]^-{\Cusim(\iota_A)} & \Cusim(A\oplus B)\ar[r]^-{\Cusim(\pi_B)}& \Cusim(B) \ar[r] &0,
}
\]
where $\alpha(x)=(x,0)$, $\beta(x,y)=y$, and $\pi_B\colon A\oplus B\to B$ is the quotient map.
A simple diagram chase using the exactness of the rows of this diagram (in the sense of Proposition \ref{splitexact} (i) and (ii)) shows that $\gamma$
is an isomorphism, as desired.
\end{proof}

\subsection{Stability}
\begin{theorem}\label{isoembed} Let $A$ be a C$^*$-algebra.
The inclusion $A\hookrightarrow A\otimes\mathcal K$ (in the top corner) induces an isomorphism
of ordered semigroups
$\Cusim(A)\to \Cusim(A\otimes\mathcal K)$.
\end{theorem}

\begin{proof}
Consider the inductive limit
\[
A\hookrightarrow M_2(A)\hookrightarrow M_4(A)\hookrightarrow \dots \hookrightarrow A\otimes \mathcal K.
\]
By the continuity of $\Cusim$ with respect to   inductive limits, it suffices to show that
the inclusion in the top corner $A\hookrightarrow M_2(A)$ induces an isomorphism
at the level of $\Cusim$. We prove this next.

Let us first assume that $A$ is unital. By Proposition \ref{unitalCusim}, the map $\Cu_{\cc}(\iota)\colon \Cu_{\cc}(A)\to \Cusim(A)$ is an isomorphism. Similarly, the inclusion of $M_2(A)$ in $M_2(A)^\sim$ induces an ordered semigroup isomorphism	$\Cu_{\cc}(\iota)\colon \Cu_{\cc}(M_2(A))\to \Cusim(M_2(A))$.
We thus have  the diagram
\[
\xymatrix{
\Cu_{\cc}(A)\ar[r]\ar[d]& 	\Cu_{\cc}(M_2(A))\ar[d]\\
\Cusim(A)\ar[r] & \Cusim(M_2(A))}
\]
where the horizontal arrows are induced by the inclusion $A\mapsto M_2(A)$	 applying the functors $\Cu_{\cc}(\cdot)$
and $\Cusim$.
It is known that $A\hookrightarrow M_2(A)$ induces an isomorphism in $\Cu$  (see \cite[Appendix 6]{Coward-Elliott-Ivanescu}). This in turn implies that the top horizontal arrow is an isomorphism from $\Cu_{\cc}(A)$ to $\Cu_{\cc}(M_2(A))$. It follows that  the bottom horizontal arrow is an isomorphism as well.

The non-unital case is reduced to the unital case as follows.
Let $A$ be a non-unital  C$^*$-algebra. Consider the diagram
\[
\xymatrix{
0\ar[r]& A \ar[r]\ar[d]  & A^\sim\ar[r]\ar[d]& \C \ar[r]\ar[d] &0\\
0\ar[r] & M_2(A) \ar[r] & M_2(A^\sim)\ar[r]& M_2(\C) \ar[r] &0,
}
\]
where the rows form short exact sequences that split and the vertical
arrows are the natural inclusions.
Applying the functor  $\Cusim$ we get
\[
\xymatrix{
0\ar[r]& \Cusim(A) \ar[r]\ar[d]  & \Cusim(A^\sim)\ar[r]\ar[d]& \Cusim(\C) \ar[r]\ar[d] &0\\
0\ar[r] & \Cusim(M_2(A)) \ar[r] & \Cusim(M_2(A^\sim))\ar[r]& \Cusim(M_2(\C)) \ar[r] &0.
}
\]
A diagram chase---as in the proof of the five lemma---using the exactness of the rows of this diagram (in the sense of Proposition \ref{splitexact} (i) and (ii)),
and that the two rightmost vertical arrows  are isomorphisms, shows that $\Cusim(A)\to \Cusim(M_2(A))$
is an isomorphism.
\end{proof}

From the stability of $\Cusim$ we derive the following proposition, which will be needed later on:
\begin{proposition}\label{almostcc}
Suppose that $A$ is $\sigma$-unital. Let $x,y\in \Cu(A)$ be such that $q(x)\leq q(y)$ in $\Cusim(A)$. Let $z\in \Cu(A)$ be a full element (i.e., such that $\infty \cdot z$ is the largest element in $\Cu(A)$). Then for every $x'\ll x$ there exist $n\in \N$ and $y'\ll y$ such that $x'+nz\leq y'+nz$.	
\end{proposition}	

\begin{proof}
Let $z=[c]$, where $c\in A\otimes\mathcal K$ is full. Let $C=\overline{c(A\otimes\mathcal K)c}$. Then $C\otimes \mathcal K\cong A\otimes\mathcal K$ by Brown's theorem. This isomorphism induces
isomorphisms  $\Cu(A)\cong \Cu(C)$ and $\Cusim(A)\cong\Cusim(C)$.
Let $\tilde x,\tilde y, \tilde z\in \Cu(C)$ denote the images of $x,y,z$ via the isomorphism of $\Cu(A)$ and $\Cu(C)$. We claim that  $q(\tilde x)\leq q(\tilde y)$ in $\Cusim(C)$. Indeed, consider  the following diagram:
\[
\xymatrix{
\Cu(A)\ar[d]^q\ar[r]&\Cu(A\otimes\mathcal K)\ar[r]\ar[d]^q&\Cu(C\otimes\mathcal K)\ar[d]^q\ar[l]& \Cu(C)\ar[l]\ar[d]^q\\
\Cusim(A)\ar[r]&\Cusim(A\otimes\mathcal K)\ar[r] &\Cusim(C\otimes\mathcal K)\ar[l]& \Cusim(C)\ar[l]
}
\]	
All squares in this diagram commute. This is  straightforward to check  for the first and third squares. The middle square commutes since the horizontal maps are induced by an isomorphism at the C*-algebra level.  It thus suffices to prove the proposition for $\tilde x,\tilde y,\tilde z$,
in $\Cu(C)$. Let $x'\ll \tilde x$. From  $q(\tilde x)\leq q(\tilde y)$ we get that 
\begin{align}\label{toproject}
x'+n[1]\leq y'+n[1]
\end{align} 
in $\Cu(\tilde C)$, for some $y'\ll \tilde y$ and some $n\in \N$. We can now ``project" all the terms of this inequality to the ideal $\Cu(C)$ in $\Cu(\tilde C)$. This operation is more explicit
in the Hilbert C*-modules picture of the Cuntz semigroup: given a Hilbert C*-module $H$ over a C*-algebra $B$ and a $\sigma$-unital closed two-sided ideal $I$, the map on the Cuntz semigroup $[H]\mapsto [HI]$ is well defined, additive, order preserving, and supremum preserving; see \cite[Section 2 and Theorem 5]{Ciuperca-Robert-Santiago}.  Applying the map $[H]\mapsto [HC]$
in the inequality \eqref{toproject}, we get
\[
x'+nz\leq y'+nz.
\] 
This proves the proposition.
\end{proof}

\section{Compact elements and functionals}\label{K0andtraces}
Here we investigate  the compact elements and functionals on the augmented Cuntz semigroup.
\subsection{Compact elements}
Let $A$ be a C*-algebra. Let $V(A)$ denote the semigroup of Murray-von Neumann classes of projections in $A\otimes\mathcal K$.  Given a projection $p\in A\otimes\mathcal K$, we denote by $[p]_{MvN}$ the element in $V(A)$ with representative $p$. Since the Murray-von Neumann equivalence relation is stronger than the relation of Cuntz equivalence, the map  $V(A)\ni [p]_{MvN}\mapsto [p]\in \Cu(A)$  is well defined. Further, this map ranges in the subsemigroup of compact elements of $\Cu(A)$. 

Recall that 
\[
K_0(A)=\{[p]_{MvN}-n[1]_{MvN}:p\in \tilde A\otimes\mathcal K,\,  \rank(p)=n\}.\] 
We thus have a map
\begin{equation}\label{K0map}
K_0(A)\ni [p]_{MvN}-n[1]_{MvN}\mapsto \ov{[p]}-n\ov{[1]}\in \Cusim(A).
\end{equation}
It is easy to show that this is a well defined additive map.

A unital C*-algebra is called stably finite if $n[1]_{MvN}+[q]_{MvN}=n[1]_{MvN}$ implies $[q]_{MvN}=0$ in $V(\tilde A)$.  If $A$ is stably finite, then $K_0(A)$ is an ordered group under the  order induced by the image of $V(A)$ in $K_0(A)$.

\begin{theorem}\label{Cusimcompacts}
Let $\Cusim_{\co}(A)$ denote the set of compact elements of $\Cusim(A)$, i.e, the set of  $x\in \Cusim(A)$ such that $x\ll x$.
\begin{enumerate}[(i)]
\item
$\Cusim_{\co}(A)$ is  a group.  
\item
If $x+y\in \Cusim_{\co}(A)$ then $x,y\in \Cusim_{\co}(A)$.
\item
If $A$ is stably finite then $\Cusim_{\co}(A)$ is isomorphic,  as an ordered group, to $K_0(A)$ via the map \eqref{K0map}. 	
\end{enumerate}
\end{theorem}	

\begin{proof}
(i) Since compact elements form a subsemigroup containing $0$, it suffices to show that every compact element has a  compact additive inverse. Let $x\in \Cusim_{\co}(A)$.  By Theorem \ref{CusimCu} (ii), there exists $y\in \Cusim(A)$
such that $x+y=0$. Since $0\ll 0$, there exists $y'\ll y$ such that $x+y'=0$.  Then $x+y=0\ll 0=x+y'$. We conclude by weak cancellation (Corollary \ref{weakcancellation}) that $y=y'$. Hence,  $y$ is compact, as desired.

(ii) Suppose that $x+y$ is compact. Choose $x'\ll x$ such that $x'+y=x+y$. By weak cancellation,
we get that $x'=x$. Hence, $x$ is compact, and similarly, $y$ is compact.

(iii) Suppose that $A$ is stably finite. Let $x\in \Cusim(A)$ be compact. Since the relation $\ll$ in $\Cusim(A)$ is the restriction of the same relation taken in $\Cu_{\cc}(\tilde A)$, $x$ is compact in   $\Cu_{\cc}(\tilde A)$. Translation by $n[1]$ in $\Cu_{\cc}(\tilde A)$ is an order isomorphism. Hence, $x=\ov{e}-n\ov{[1]}$, where $\overline{e}\in \Cu_{\cc}(\tilde A)$ is compact and lifts to $e\in \Cu(\tilde A)$. By Lemma \ref{ovone}, there exists $p\in \tilde A\otimes\mathcal K$ such that $e=[p]$.  It follows  that $x=\ov{[p]}-n\ov{[1]}$, i.e., $x$ is in the range of the map \eqref{K0map}. 

Let us show that the map \eqref{K0map} is an order embedding.  Suppose that $\ov{[p]}-n\ov{[1]}\leq \ov{[q]}-\ov{n[1]}$, where $p,q\in \tilde A\otimes \mathcal K$ are projections  of rank $n$. Then $\ov{[p]}\leq \ov{[q]}$.  Since $[p]\ll [p]$,  this implies that  $[p]+m[1]\leq [q]+m[1]$ for some $m$. Cuntz subequivalence of projections agrees with Murray-von Neumann subequivalence. We thus get that $[p]_{MvN}+m[1]_{MvN}\leq [q]_{MvN}+m[1]_{MvN}$, which in turn implies that $[p]_{MvN}-n[1]_{MvN}\leq [q]_{MvN}-n[1]_{MvN}$, as desired.
\end{proof}

\subsection{Functionals}
Let $S$ be a positively ordered Cu-semigroup. A map $\lambda\colon S\to [0,\infty]$ is called a functional if it maps 0 to 0, it is additive, order preserving, and preserves the suprema of increasing sequences. It is called  densely finite  if $\lambda(x)<\infty$ whenever $x\ll y$ for some $y\in S$. If $S=\Cu(A)$, then the densely finite functionals are in bijective correspondence with the densely finite 2-quasitraces on $A_+$ as well as with the densely finite rank functions (\cite{Elliott-Robert-Santiago}).

Let $\lambda\colon \Cusim(A)_+\to [0,\infty]$ be a densely finite functional. 
Our goal is to extend it to all of $\Cusim(A)$. Let $x\in \Cusim(A)$. We know, by Theorem \ref{CusimCu} (i),  that there  exists $z\geq 0$ such that $x+z\geq 0$. Moreover,  since $0\ll 0$, we can choose  $z$ such that $z\ll z'$ for some $z'$.
Then $\lambda$ is finite on $z$. Let us define
\begin{equation}\label{lambdadef}
\lambda(x):=\lambda(x+z)-\lambda(z)\in(-\infty,\infty].
\end{equation}
Clearly, this is the only possible way to extend $\lambda$ additively to $\Cusim(A)$.

\begin{lemma}
Let $\lambda\colon \Cusim(A)_+\to [0,\infty]$ be a densely finite functional. The extension of $\lambda$ to $\Cusim(A)$ defined above is additive, order preserving, and preserves the suprema of increasing sequences. 
\end{lemma}		
\begin{proof}
Let 
\begin{equation}\label{defP}
P=\{z\in \Cusim(A)_+: z\ll z'\hbox{ for some $z'$}\}.
\end{equation}
Let $x\in \Cusim(A)$. Let us show first that $\lambda(x)$, defined as in \eqref{lambdadef}, is independent of the choice of $z$, as long as $x+z\geq 0$ and $z\in P$.
Suppose that $z'$ is another element with these properties. Then 
\[
\lambda(x+z)+\lambda(z')=\lambda(x+z+z')=\lambda(x+z')+\lambda(z).
\]	
Hence, $\lambda(x+z)-\lambda(z)=\lambda(x+z')-\lambda(z')$.

Suppose that $x\leq y$ in $\Cusim(A)$. Choose $z\in P$ such that  $x+z\geq 0$. 
Then $x+z\leq y+z$ and so 
\[
\lambda(x)=\lambda(x+z)-\lambda(z)\leq \lambda(y+z)-\lambda(z)=\lambda(y).
\] 
Thus, $\lambda$ is order preserving.

Let $(x_n)_n$ be an increasing sequence in $\Cusim(A)$. Choose $z$
such that $x_1+z\geq 0$ and $z\in P$. Then 
\[
\sup \lambda(x_i)=\sup\lambda(x_i+z)-\lambda(z)=\lambda(\sup x_i+z)-\lambda(z)=\lambda(\sup x_i).
\] 

Additivity of the extension of $\lambda$ is handled similarly. 
\end{proof}	

Let $\ov{\R}:=\R\cup\{\infty\}$. Let $F_0(\Cusim(A))$ denote the set of $\lambda\colon \Cusim(A)\to \ov{\R}$ that are additive, preserve 0, sequential suprema, and are densely finite. 
We endow $F_0(\Cusim(A))$  with the topology such that a net $(\lambda)_i$ converges to $\lambda$ if 
\[
\limsup \lambda_i(x)\leq \lambda(y)\leq \liminf \lambda_i(y) \hbox{ for all }0\leq x\ll y.
\]

Given $x\in \Cusim(A)$ we get a  function $\hat x\colon F_0(\Cusim(A))\to \ov{\R}$ defined by  $\hat x(\lambda)=\lambda(x)$ for all $\lambda\in F_0(\Cusim(A))$. 

\begin{lemma}\label{hatx}
The function $\hat x$ is linear and lower semicontinuous.
\end{lemma}	
\begin{proof}
Let $\lambda_i\to \lambda$. Choose $z\in P$ such that $x+z\geq 0$ and $z\geq 0$, where $P$
is the set defined in \eqref{defP}. Choose $0\ll z'\ll z$ such that $x+z'\geq 0$ (it exists by the compactness of 0). Then 
\begin{align*}
\lambda(x+z')-\lambda(z) &\leq \liminf \lambda_i(x+z')-\limsup \lambda_i(z')\\
&=\liminf \lambda_i(x).
\end{align*}	
Passing to the supremum over all $z'\ll z$ on the left side we get that $\lambda(x)\leq \liminf \lambda_i(x)$. That is, $\hat x$ is l.s.c.
\end{proof}

Recall that the map $q\colon \Cu(A)\to \Cusim(A)_+$ defined as $q(x)=\ov{x}$ is an onto Cu-morphism (Lemma \ref{ontoq}). It follows that $\lambda\mapsto \lambda\circ q$ 
is an embedding  of $F_0(\Cusim(A))$ into $F_0(\Cu(A))$. The topology on $F_0(\Cusim(A))$ that we have defined above is precisely the one induced by the topology on $F_0(\Cu(A))$ via this embedding.

\begin{lemma}
Suppose that $A$ is $\sigma$-unital. Let $\lambda\colon \Cu(A)\to [0,\infty]$ be a functional that is finite on a full element of $\Cu(A)$. Then $\lambda$ factors through $q$, i.e.,
$\lambda=\tilde \lambda \circ q$, where $\tilde\lambda\in F_0(\Cusim(A))$.	
\end{lemma}	
\begin{proof}
Let $z\in \Cu(A)$ be a full element such that $\lambda(z)<\infty$.	
Let us shows that if $x,y\in \Cu(A)$ are such that $\ov x\leq \ov y$ in $\Cusim(A)$ then $\lambda(x)\leq \lambda(y)$.
Let $x'\ll x$. 	By Proposition \ref{almostcc}, $x'+nz\leq y+nz$ for some $n$.  It follows that $\lambda(x')\leq \lambda(y)$. Passing to the supremum over all $x'\ll x$, we get that $\lambda(x)\leq \lambda(y)$. We can thus define $\tilde \lambda(\ov{x})=\lambda(x)$, for $x\in \Cu(A)$. Additivity, preservation of order,  and preservation of 0, are readily passed from $\lambda$ to $\tilde\lambda$. Let us show that $\tilde \lambda$ preserves sequential suprema: Suppose that $(x_n)_{n}$ is an increasing sequence in $\Cusim_+(A)$. By Theorem \ref{suprema}, there exists an increasing sequence $z_n\in \Cu(\tilde A)$ such that $\ov{z_n}\leq x_n$ for all $n$ and $\ov{\sup z_n}=\sup x_n$. Since $0\leq \rank(z_n)\leq \rank(x_n)=0$, we must have that $\rank(z_n)=0$ for all $n$, i.e., $z_n\in \Cu(A)$. Then 
\[
\tilde \lambda(\sup x_n)=\lambda(\sup z_n)=\sup \lambda(z_n)\leq \sup \tilde\lambda(x_n).
\]  
Hence,  $\tilde \lambda(\sup x_n)\leq \sup \tilde\lambda(x_n)$. The opposite inequality is 
straightforward. Thus, $\tilde\lambda$ preserves sequential suprema. We have thus defined 
a densely finite functional $\tilde\lambda$ on $\Cusim(A)_+$ such that $\lambda=\tilde\lambda q$. Extending it to $\Cusim(A)$ as in \eqref{lambdadef}, we obtain a functional in $F_0(\Cusim(A))$.
\end{proof}

Let $\lambda$ be a functional on $\Cu(A)$ meeting the assumptions of the previous theorem, so that it arises as $\lambda=\tilde\lambda \circ q$, for $\tilde\lambda\in F_0(\Cusim(A))$. Let 
$x\in \Cusim(A)$. Then we can define a pairing of $\lambda$ and $x$: 
\[
\langle \lambda, x\rangle := \tilde\lambda(x).
\]
If every densely functional on $\Cu(A)$ factors through $q$, then we can define this pairing 
on all $F_0(\Cu(A))\times \Cusim(A)$. This is the case when the primitive spectrum of $A$ is compact. We summarize the situation in the following theorem:  

\begin{theorem}\label{functsummary}
Let $A$ be a $\sigma$-unital C*-algebra with compact primitive spectrum. The following are true:
\begin{enumerate}[(i)]
\item
The densely finite functionals on $\Cusim(A)$ are in bijection with the densely finite functionals on $\Cu(A)$ via the map $\lambda\mapsto \lambda\circ q$.

\item
Given $x\in \Cusim(A)$ the function $\hat x\colon F_0(\Cu(A))\to \ov{\R}$ defined as 
\[
\hat x(\lambda)=\langle \lambda, x\rangle\hbox{ for all }\lambda\in F_0(\Cu(A))
\]
is additive and lower semicontinuous.
\end{enumerate}
\end{theorem}
\begin{proof}
Let $z=[a]\in \Cu(A)$ be the class of a strictly positive element of $A$. Then $z$, is full, i.e,
$\infty \cdot z$ is the largest element of $\Cu(A)$. Choose a rapidly increasing sequence $(z_n)_n$ with supremum $z$. Since  the primitive spectrum of $A$ is compact, $z_n$ is full
for large enough $n$. We thus find $z_n$ that is full and such that every densely finite functional on $\Cu(A)$ is finite on $z_n$.  It follows that every densely finite functional on $\Cu(A)$ factors through $q$ by the previous lemma. This proves (i). (ii) follows from Lemma \ref{hatx} and the agreement of the topologies on $F_0(\Cu(A))$ and $F_0(\Cusim(A))$.
\end{proof}

\subsection{Simple pure C*-algebras}
A C*-algebra is said to be pure if its Cuntz semigroup is almost unperforated and almost divisible (\cite[Definition 3.6 (i)]{WinterPure}). Let us recall the definition of these properties: A $\Cu$-semigroup $S$ is called almost unperforated if $(k+1)x\leq ky$ implies  $x\leq y$ for all $x,y\in S$ and $k\in \N$. Almost unperforation is equivalent to the property of strict comparison, defined as follows: for all $x,y\in \Cu(A)$ and $\epsilon>0$ if $\lambda(x)\leq (1-\epsilon)\lambda(y)$ for all $\lambda\in F(S)$ then $x\leq y$. $S$ is called almost divisible if for all $x',x\in S$ with $x'\ll x$, and $n\in \N$, there exists $y\in S$ such that $ny\leq x$ and $x'\leq (n+1)y$. C*-algebras that tensorially absorb the Jiang-Su C*-algebra are pure (\cite{RordamZ}). 

Let us denote the cone of densely finite functionals $F_0(\Cu(A))$ simply by $Q$. Recall that $Q$ may also be regarded as the cone of densely finite 2-quasitraces on $A$. \emph{For the remainder of this section $A$ denotes a simple, separable, pure C*-algebra such that $Q\neq \{0\}$.} Our goal is to calculate $\Cusim(A)$. We accomplish this in  Theorem \ref{simpleCusim} below. This calculation, under the additional assumption that $A$ has stable rank one, is obtained in \cite{Robert}. We follow closely the same arguments, while circumventing the stable rank one assumption.

\begin{lemma}
$\Cusim(A)$ is simple in the sense that for every nonzero $z\in \Cusim(A)_+$ we have that $\infty\cdot z$ is the largest elements of $\Cusim(A)$.
\end{lemma}

\begin{proof}
The subsets of $\Cu(A)$ that are closed under addition and sequential suprema are in bijection with the closed two-sided ideals of $A$ (\cite{CiupercaThesis}). Since $A$ is simple, the lemma is true for $\Cu(A)$: if $z\in \Cu(A)$ is nonzero then $\infty\cdot z$ is the largest element of $\Cu(A)$. But $\Cu(A)$ is mapped onto $\Cusim(A)_+$ by the Cu-morphism $q$ (Lemma \ref{ontoq}). This readily implies the same property for $\Cusim(A)$.
\end{proof}

Let us call an element $x\in \Cusim(A)$ soft if for each $x'\ll x$ there exists a nonzero $z\in \Cusim(A)_+$ such that $x'+z\leq x$. 
\begin{lemma}
The soft elements form an absorbing subsemigroup: if $x,y\in \Cusim(A)$ and $y$ is soft then $x+y$ is soft. 
\end{lemma}
\begin{proof}
If $w\ll x+y$
then $w\leq x+y'$ for some $y'\ll y$. Find $z\geq 0$ nonzero such that $y'+z\leq y$. Then $w+z\leq x+y'+z\leq x+y$.
\end{proof}

\begin{lemma}\label{softorcompact}
Each $x\in \Cusim(A)$ is either soft or compact and not both.
\end{lemma}	
\begin{proof}
Let $x'\ll x$. Choose $x''\in \Cusim(A)$ such that $x'\ll x''\ll x$. Find $z\geq 0$ such that $x'+z\leq x\leq x''+z$ (Theorem \ref{CusimCu} (ii)). If $z=0$ then $x=x''\ll x$, so $x$ is compact. If it is the case that $z\neq 0$ for all $x'\ll x$ then $x$ is soft.

If $x$ is both soft and compact, then $x+z=x$, with $z$ positive and nonzero. Since  densely finite functionals  are finite on compact elements, $\lambda(z)=0$ for any $\lambda\in Q$. By the simplicity of $\Cusim(A)$, $z$ is full. Hence $\lambda=0$, which contradicts our assumption that $Q\neq \{0\}$.
\end{proof}

Let $\mathrm{Lsc}(Q,\ov{\R})$ denote the set of functions $f\colon Q\to \ov \R$ that are lower semicontinuous, linear, and map 0 to 0. We have shown in Theorem \ref{functsummary} that $\hat x\in \mathrm{Lsc}(Q,\ov{\R})$ for every $x\in \Cusim(A)$. Let $\mathrm{Lsc}_{++}(Q)$ denote the
functions in  $\mathrm{Lsc}(Q,\ov{\R})$ that are strictly positive on the nonzero functionals.
By the simplicity of $\Cusim(A)$, if $x\in \Cusim(A)_+$ is nonzero then $\hat x\in \mathrm{Lsc}_{++}(Q)$.

We will make use of the calculation of $\Cu(A)$ (see \cite[Theorem 6.2]{Tikuisis-Toms}). Here are the main facts that we will need:
\begin{description}
\item[Fact 1]
Every element of $\Cu(A)$ is either compact or soft (same definition as above).
\item[Fact 2]
The map $x\mapsto \hat x$ is an isomorphism of the semigroup of soft elements in $\Cu(A)$ with 
$\mathrm{Lsc}_{++}(Q)$.
\end{description} 

\begin{lemma}
The map $x\mapsto \hat x$ is an ordered semigroup isomorphism from the set of positive soft elements in $\Cusim(A)_+$ to  $\mathrm{Lsc}_{++}(Q)$. 
\end{lemma}
\begin{proof}
Recall that $\hat x=\widehat{q(x)}$ for all $x\in \Cu(A)$. Further, by Fact 2 above, $x\mapsto \hat{x}$ is an isomorphism from the soft elements of $\Cu(A)$ to $\mathrm{Lsc}_{++}(Q)$. It thus suffices to show that  $q\colon \Cu(A)\to \Cusim(A)_+$ is an ordered semigroup isomorphism from the subsemigroup of soft element of $\Cu(A)$ to the subsemigroup of positive soft elements in $\Cusim(A)$. We prove this next.

Suppose that $y\in \Cu(A)$ is soft. Let $x'\ll q(y)$. Choose $y'\ll y$ such that $x'\leq q(y')$ (recall that $q$ is a Cu-morphism). Choose a nonzero $w\geq 0$ such that $y'+w\leq y$. Then
$x'+q(w)\leq q(y)$. Further, we cannot have   $q(w)=0$, for otherwise $\hat w=0$, which implies that $w=0$ (by strict comparison of $\Cu(A)$). Thus, $q$ maps soft elements of $\Cu(A)$ to soft elements of $\Cusim(A)$.

To see that $q$ is an order embedding on the soft elements, recall that every functional on $\Cu(A)$ factors through $q$. Hence, $\widehat{q(y)}=\hat y$ for all $y\in \Cu(A)$. Further, $\Cu(A)\ni y\mapsto \hat y\in \mathrm{Lsc}_{++}(Q)$ is an order isomorphism on the soft elements (Fact 2 recalled above). Thus, $q$ is an order embedding. 

Finally, let us prove surjectivity: any $x\in \Cusim(A)_+$ has a lift in $\Cu(A)$.
If the lift is compact, then $x$ is compact, since $q$ preserves $\ll$. Thus, soft elements in $\Cusim(A)$ lift to soft elements in $\Cu(A)$.
\end{proof}

\begin{lemma}\label{softiso}
Let $\Cusim(A)_{\mathrm{sft}}$ denote the set of soft elements in $\Cusim(A)$. 
\begin{enumerate}[(i)]
\item
For all $x,y\in \Cusim(A)$, with $x$ soft, $\hat x\leq \hat y$ implies $x\leq y$.
\item
The map $x\mapsto \hat x$ is an isomorphism from $\Cusim(A)_{\mathrm{sft}}$ to $\mathrm{Lsc}(Q,\ov{\R})$.	
\end{enumerate}	
\end{lemma}	

\begin{proof}
(i) Let $x,y\in \Cusim(A)$ be  such that $\hat x\leq \hat y$ and $x$ is soft.
Choose a sequence of functions $f_n\in \mathrm{Lsc}_{++}(Q)$ that are finite-valued, continuous, and   with supremum the function equal to infinity on all $Q$ except 0. By the isomorphism of the set of positive soft elements in $\Cusim(A)$ with $\mathrm{Lsc}_{++}(Q)$, there exists an increasing sequence of positive soft elements $(z_n)_n$ in $\Cusim(A)_+$ such that $\hat z_n=f_n$. Since 
\[
x+\infty =y+\infty=\infty\geq 0
\] 
and  $0\in \Cusim(A)$ is compact, there exists $n$ such that $x+z_n\geq 0$ and $y+z_n\geq 0$. Set $z=z_n$. Then $\hat x+\hat z\leq \hat y+\hat z$. Since $x+z$ and $y+z$ are soft and positive, we have $x+z\leq y+z$ in $\Cusim_+(A)$.  Let $x'\ll x$. Let us show that $x'+z\ll x+z$. Choose $x''$ such that $x'\ll x''\ll x$. Since $x$ is soft, there exists  a nonzero $w\geq 0$ such that $x''+w\leq x$. Then $x''+w+z\leq x+z$. In $\mathrm{Lsc}_{++}(Q)$, we have  $\hat z\ll \hat w+\hat z$, since $\hat w$ is strictly positive  and $\hat z$ is continuous and finite. It follows that $z\ll z+w$ in $\Cusim(A)$. Hence, 
\[
x'+z\ll x''+(w+z)\leq x+z\leq y+z. 
\] 
By weak cancellation (Corollary \ref{weakcancellation}), $x'\leq y$.  Passing to the supremum over all $x'\ll x$, we conclude that $x\leq y$.

(ii) It is clear from (ii) that $x\mapsto \hat x$ is an order embedding. Let us prove surjectivity. We have already shown that the functions in $\mathrm{Lsc}_{++}(Q)$ are in the range of this map. Any function in  $\mathrm{Lsc}(Q, \ov{\R})$ can be expressed in the form $g-f$, where 
$f,g\in \mathrm{Lsc}_{++}(Q)$ and $f$ is continuous and finite.  (If  $h\in \mathrm{Lsc}(Q, \ov{\R})$, then $h$ is bounded from below  by lower semicontinuity and the compactness of $Q$. Choose $f_0\in \mathrm{Lsc}_{++}(Q)$  continuous. Then the desired decomposition is $h=(h+Cf_0)-Cf_0$, with $C>0$ sufficiently large.) It thus suffices to show that if $f\in \mathrm{Lsc}_{++}(Q)$ is continuous and finite  then $-f$ is in the range of the map $x\mapsto \hat x$.  Let $\epsilon>0$. Choose $x,y\in \Cusim(A)$ positive, soft, such that $\hat x=f$ and  $\hat y=\epsilon f$.  We have $(1+\epsilon)f\ll (1+2\epsilon)f$ in $\mathrm{Lsc}_{++}(Q)$ (since $f$ is continuous, finite, and strictly positive). Hence  $x+y\ll x+2y$. Using Theorem \ref{CusimCu} (ii), choose $w\in \Cusim(A)$ such that 
\[
x+y+w\leq 0\leq x+2y+w.
\] 
Set $z=y+w$. Then $z$ is soft, since $y$ is, and  
\[
-(1+2\epsilon)f\leq \hat z\leq -(1+\epsilon)f.
\] 
Applying the construction of $z$ for  $\epsilon_n=1/2^n$, with $n=1,2,\ldots$, we obtain a sequence $(z_n)_n$  of soft elements such that $\hat z_n$ is increasing
and has supremum $-f$. Since the $z_n$ are soft, the sequence is increasing in $\Cusim(A)$ (by (i)). Letting $z=\sup z_n$ we have $\hat z=-f$, as desired.
\end{proof}	

Let $\Cusim(A)_{\co}$ and $\Cusim(A)_{\mathrm{sft}}$ denote the subsemigroups of compact and soft elements respectively. By Lemma \ref{softorcompact}, $\Cusim(A)$ is the disjoint union of these two subsemigroups. By Theorem \ref{Cusimcompacts},  $\Cusim(A)_{\co}$ is isomorphic to $K_0(A)$ via the map \eqref{K0map}. By Lemma \ref{softiso},  $\Cusim(A)_{\mathrm{sft}}$ is isomorphic to  $\mathrm{Lsc}(Q,\ov{\R})$ via the map $x\mapsto \hat x$, where $Q=F_0(\Cu(A))$. Let us define a bijection
\[
\Cusim(A)\rightarrow  K_0(A)\sqcup \mathrm{Lsc}(Q,\ov{\R})
\]
by combining these two isomorphisms:
\begin{equation}\label{isomorphism}
\begin{aligned}
\ov{[p]}-\ov{[q]} &\mapsto [p]_{MvN}-[q]_{MvN}\in K_0(A)\hbox{ if }\ov{[p]}-\ov{[q]}\in \Cusim_{\co}(A)\\
x &\mapsto \hat x\in \mathrm{Lsc}(Q,\ov{\R})\hbox{ if }x\in \Cusim_{\mathrm{sft}}(A).
\end{aligned}
\end{equation}

We obtain a map from $K_0(A)$ to $\mathrm{Lsc}(Q,\ov{\R})$ by first regarding $x\in K_0(A)$
as a compact element in $x\in \Cusim(A)$ and then finding $\hat x\in \mathrm{Lsc}(Q,\ov{\R})$.
Let us continue to denote this map with a hat: $K_0(A)\ni x\mapsto \hat x\in \mathrm{Lsc}(Q,\ov{\R})$. We can now   endow $K_0(A)\sqcup \mathrm{Lsc}(Q,\ov{\R})$ with an order and an addition operation as follows:  On the sets $K_0(A)$ and $\mathrm{Lsc}(Q,\ov{\R})$, these are already defined. Let $x\in K_0(A)$ and $f\in \mathrm{Lsc}(Q,\ov{\R})$. We define 
\[
x+f:=\hat x+f\in \mathrm{Lsc}(Q,\ov{\R}).\] 
We define $f \leq x$ if $f\leq \hat x$ and  $x \leq f$ if $\hat x+h=f$ for some $h\in \mathrm{Lsc}_{++}(Q)$.

\begin{theorem}\label{simpleCusim}
Let $A$ be a simple, separable, pure C*-algebra such that $F_0(\Cu(A))\neq \{0\}$. Then the map defined in \eqref{isomorphism}  is an isomorphism of ordered semigroups.
\end{theorem}	

\begin{proof}
We have already shown that this map is bijective and an isomorphism when restricted both to the subsemigroups of compact elements and of soft elements. It remains to show that it is additive and an order embedding. 

Let $x,y\in \Cusim(A)$ with $x$ soft and $y$ compact. Then $x+y$ is soft. So the map is additive by the definition of addition on $K_0(A)\sqcup \mathrm{Lsc}(Q,\ov{\R})$. 

Suppose that the image of $x$ is less than or equal to the image of $y$, i.e., $\hat x\leq \hat y$. Then, by Lemma \ref{softiso}, $x\leq y$. 

Suppose on the other hand that the image of $y$ is less than the image of $x$. By the definition of the order in the codomain, this means that $\hat y+h=\hat x$, where  $h\in \mathrm{Lsc}_{++}(Q)$. Let $z\in \Cusim(A)_+$ be a soft element such that $\hat z = h$. Then $y+z$ is soft and $\hat y+\hat z=\hat x$. Hence, $y+z=x$, by Lemma \ref{softiso}. Since $z\geq 0$, we have that $y\leq x$. 
\end{proof}	

\begin{example}
Let $\mathcal W$	denote the Jacelon-Razak C*-algebra (\cite{Jacelon}).  Then $\Cu(\mathcal W)\cong [0,\infty]$ and $K_0(\mathcal W)=\{0\}$. Thus, $\Cusim(\mathcal W)\cong \{0\}\sqcup\ov{\R}$. Notice that the neutral element is $0\in \{0\}$ and not $0\in \ov{\R}$, which is soft.  
\end{example}

The calculation of $\Cusim(A)$ in Theorem \ref{simpleCusim} applies to simple $\mathcal Z$-stable, projectionless C*-algebras. Moreover, in this case $\Cusim(A)$ agrees with the original
construction in \cite{Robert}, since these C*-algebras have finite stable rank (see Subsection \ref{cforiginal}), as we now show.

\begin{theorem}
Let $A$ be a simple, projectionless,  $\mathcal Z$-stable C*-algebra. Then the stable rank of 
$A$ is at most two.
\end{theorem}

\begin{proof}
By \cite[Corollary 3.2]{RobertGlasgow}, $A$ almost has stable rank one, in the sense that $A$ is contained in the closure of the invertible elements of $\tilde A$. Let us show that if $A$ almost has stable rank one  then it has stable rank at most two. Let $(\alpha_1\cdot 1+a_1,\alpha_2\cdot 1+a_2)\in \tilde A \times \tilde A$, where  $\alpha_1,\alpha_2\in \C$ and $a_1,a_2\in A$. Our goal is to show that they are in the closure of the left invertible pairs in  $\tilde A \times \tilde A$. By a small perturbation, we may assume without loss of generality  that $(\alpha_1,\alpha_2)\neq (0,0)$. Multiplying the vector $(\alpha_1\cdot 1+a_1,\alpha_2\cdot 1+a_2)$ by a suitable invertible scalar matrix, we may further assume that  $(\alpha_1,\alpha_2)=(1,0)$. The pair now has the form $(1+a_1,a_2)$, with $a_1,a_2\in A$. But, by the almost stable rank one property, $a_2$ is in the closure of the invertible elements in $\tilde A$. Let $\tilde a_2\in \tilde A$ be invertible. Then $(1+a_2,\tilde a_2)$ is left invertible, with left inverse $(0,\tilde a_2^{-1})$. This proves the theorem.
\end{proof}

\end{document}